\documentclass{article}
\usepackage[all]{xy}
\usepackage{authblk}

\usepackage{graphicx}
\usepackage{multirow}
\usepackage{amsmath,amssymb,amsfonts}
\usepackage{amsthm}
\usepackage{mathrsfs}
\usepackage[title]{appendix}
\usepackage{xcolor}
\usepackage{textcomp}
\usepackage{manyfoot}
\usepackage{booktabs}
\usepackage{algorithm}
\usepackage{algorithmicx}
\usepackage{algpseudocode}
\usepackage{listings}

\theoremstyle{plain}
\newtheorem{theorem}{Theorem}[section]
\newtheorem{proposition}[theorem]{Proposition}
\newtheorem{lemma}[theorem]{Lemma}
\newtheorem{corollary}[theorem]{Corollary}

\theoremstyle{remark}
\newtheorem{remark}[theorem]{Remark}

\theoremstyle{definition}
\newtheorem{definition}[theorem]{Definition}
\newtheorem{example}[theorem]{Example}
\newtheorem*{claim*}{Claim}

\title{Stable Lyapunov Spectrum Rigidity of Nilmanifold Endomorphisms}
\author[1]{Ruihao Gu}
\author[2]{Wenchao Li \thanks{Corrosponding author: lwc@pku.edu.cn}}

\affil[1]{Shanghai Center for Mathematical Sciences, Fudan University, Shanghai, 200433, China}
\affil[2]{School of Mathematical Sciences, Peking University, Beijing, 100871, China}

\date{}

\begin{document}








\maketitle

{\bf \centering Abstract

}

Under some non-invertibility and irreducibility condition, for nilmanifold Anosov maps with one-dimensional stable bundle, we get the equivalence among the existence of invariant unstable bundle, the existence of topological conjugacy to its linear part, and a constant periodic stable Lyapunov exponent.\\



\section{Introduction}\label{Introduction}
As an important example of dynamical systems, Anosov diffeomorphisms have been concerned since several decades ago. Since then, many results about topological classification of Anosov diffeomorphisms have been established. For example, Anosov diffeomorphisms are structurally stable \cite{Anosov67}, and nilmanifold Anosov diffeomorphisms are always topologically conjugate to hyperbolic nilmanifold automorphisms \cite{Franks1970,Manning1974,Newhouse1970}. Note that the conjugacy between two Anosov diffeomorphisms is actually H\"{o}lder continuous, but generally not smooth \cite{Katok1995}.

However, things become different when it comes to non-invertible, non-expanding Anosov maps. They are \textit{not} structurally stable \cite{ManePugh1975,Przytycki1976}, and a toral Anosov map may not be topologically conjugate to any toral endomorphism, when it has no invariant unstable bundle, see Proposition \ref{intro prop: special and conjugacy}. Reasonably, rigidity phenomenon happens when the conjugacy indeed exists. For example, the conjugacy between two non-invertible, non-expanding Anosov maps on a 2-torus is automatically smooth along each stable leaf \cite{An2023,GS22}.

On the Smale's conjecture \cite{Smale1967} that an Anosov diffeomorphism is always supported on an infra-nilmanifold, which is finitely covered by a nilmanifold, and since all known examples of Anosov maps are indeed supported on infra-nilmanifolds, it is natural to study nilmanifold Anosov maps. However, most researches on \textit{rigidity} issue are focus on tori, see for example \cite{An2023,Llave1987,Llave1992,GG2008,Gogolev2008,Gogolev2017,GKS2011,GKS2020,GS22,RY2019}. 

As far as authors know, the study of rigidity phenomenon for Anosov maps on non-toral nilmanifolds can be only found in \cite{DeWitt2021,GH2023}. Dewitt \cite{DeWitt2021} studied the local Lyapunov spectrum rigidity of hyperbolic nilmanifold automorphisms under some irreducibility and sorted spectrum condition, which is partially related to the conditions we need in this paper.

It is worth to point out that the promotion from tori to nilmanifolds is nontrivial. The lack of commutativity leads to a weird geometric structure and a more complicated algebraic structure. To overcome such obstructions, generally an induction with respect to the lower central series is needed, although some properties may be destroyed during the induction.

\subsection{Rigidity of Conjugacy}
In this paper, inspired by \cite{An2023}, we consider a rigidity question of non-invertible nilmanifold Anosov maps: 

\textbf{Question.} \textit{Is the conjugacy between a non-invertible Anosov map and its linear part automatically smooth along each stable leaf\ ?}

The \textit{linear part} of a covering map $f$ on a nilmanifold $M$, is the unique (up to homotopy) endomorphism $\Psi$ such that $f$ is homotopic to $\Psi$. Note that if $f$ is Anosov, then $\Psi$ is hyperbolic and unique up to an algebraic conjugacy \cite{Aoki1994,Sumi1996}.

Before answering this question, let us talk about a main tool on this question in \cite{An2023}, the \textit{exponentially dense preimage set}. In the torus case, \cite{An2023} shows that a non-invertible irreducible toral endomorphism $\Psi: \mathbb{T}^d\to \mathbb{T}^d$ has exponential density of preimage set, which means that the set of $k$-th preimages of any point becomes dense exponentially as $k$ tends to infinity, i.e., there exist constants $C > 1$ and $0 < \mu < 1$ such that for every point $x\in \mathbb{T}^d$, the set $\Psi^{-k}(x)$ is $C\mu^k$-dense in $\mathbb{T}^d$.

Notice that irreducibility is not a necessary condition. For example, the endomorphism $A = \left(\begin{array}{cc}2 & 0 \\ 0 & 2\end{array}\right)$ of $\mathbb{T}^2$ still has exponential density of preimage set. We generalize this result for nilmanifold endomorphisms as following. A nilmanifold endomorphism is said to be \textit{totally non-invertible}, if its eigenvalues are not algebraic units. In some sense the definition can be understood as having no invertible factors, see the discussion after Definition \ref{def: totally non-invertible}.

\begin{theorem}\label{intro thm: exponential density}
	Let $\Psi: M \to M$ be a nilmanifold endomorphism. Then $\Psi$ has exponential density of preimage set if and only if $\ \Psi$ is totally non-invertible.
\end{theorem}

We will characterize the exponential density of preimage set more completely in Theorem \ref{thm: exponential density}.

Next, we show global stable Lyapunov spectrum rigidity for special nilmanifold Anosov maps with one-dimensional stable bundle.

\begin{theorem}\label{intro thm: conjuacy implies pdc}
	Let $f$ be an Anosov map on a nilmanifold $M$ with one-dimensional stable bundle and totally non-invertible linear part $\Psi$. If $f$ is topologically conjugate to $\Psi$, then the stable Lyapunov exponent of every periodic point of $f$ coincides with $\Psi$.
\end{theorem}

\begin{remark}
	Note that Theorem \ref{intro thm: conjuacy implies pdc} only needs $C^1$-regularity of the map $f$.
\end{remark}

As a corollary, if the topological conjugacy exists, then it is automatically smooth along each stable leaf. Here we assume $C^r$-regularity ($r > 1$) of $f$, for applying Livschitz Theorem, see Proposition \ref{prop: Livschitz}. 

\begin{corollary}\label{intro cor: rigidity of conjugacy}
	Let $f$ be a $C^r$ ($r > 1$) Anosov map on a nilmanifold $M$ with one-dimensional stable bundle and totally non-invertible linear part $\Psi$. If $f$ is topologically conjugate to $\Psi$ via some homeomorphism $h$, then $h$ is $C^r$-smooth along each stable leaf.
\end{corollary}

On the other hand, we prove the opposite direction, where $\Psi$ demands a dense stable leaf instead of exponential density of preimage set.

\begin{theorem}\label{intro thm: pdc implies conjugacy}
	Let $f$ be a $C^r$ ($r > 1$) Anosov map on a nilmanifold $M$ with one-dimensional stable bundle and horizontally irreducible linear part $\Psi$. If the stable Lyapunov exponent of every periodic point of $f$ coincide with $\Psi$, then $f$ is topologically conjugate to $\Psi$.
\end{theorem}

Recall that in \cite{An2023} dealing with the torus case, in order for dense stable leaves, the linear part $\Psi$ needs to be \textit{irreducible}, i.e., the characteristic polynomial of $\Psi$ is irreducible over $\mathbb{Q}$. In the nilmanifold case, the condition becomes \textit{horizontal irreducibility}, in order for the same property. An endomorphism $\Psi$ of a nilmanifold $M = N/\Gamma$ is horizontally irreducible, if the induced toral endomorphism $\Psi_1$ of the horizontal torus $M_1 = (N/N_2)/(\Gamma/\Gamma_2)$ is irreducible. Here $\Gamma$ is a lattice of a simply connected nilpotent Lie group $N$, $N_2 = [N, N]$, $\Gamma_2 = \Gamma\bigcap N_2$.

\subsection{Rigidity of existence of unstable bundle}
The most studies \cite{Aoki1994,Moosavi2019,Moosavi2019a,Sumi1996}, in the past, on the existence of conjugacy between a nilmanifold Anosov map $f$ with its linear part $\Psi$, are focus on a direct criterion: the \textit{existence of $f$-invariant unstable bundle}, see also Proposition \ref{intro prop: special and conjugacy}. For short, we call such $f$ \textit{special}.

Indeed, in \cite{ManePugh1975}, for a given special Anosov map $f$ of any closed manifold $M$, Ma\~n\'e and Pugh $C^r$-smoothly perturbs it along stable leaf such that there is no conjugacy (close to identity) between $f$ and the perturbation $g$, meanwhile $g$ is not special. It follows that being special is not a $C^r$-open property. In \cite{Przytycki1976}, Przytycki even constructs a class of Anosov maps which has infinitely many unstable directions on certain points such that unstable directions on a certain point contains a curve homeomorphic to an interval in the ${\rm dim}(TM/E^s)$-Grassmann space, where $E^s$ is the stable bundle. Moreover, this phenomenon observed by Przytycki is generic \cite{CM2022,MT2016,Moosavi2019a}.

We already know that on a torus \cite{Aoki1994,Sumi1996,Moosavi2019}, an Anosov map is special if and only if it is topologically conjugate to its linear part. Although researches \cite{Moosavi2019} proved the same conclusion for nilmanifolds, most notably, some claims in \cite{Moosavi2019} need more conditions than stated in their statements. For instance, in \cite[Proposition 3.6]{Moosavi2019}, they need the following claim \cite[Lemma2.21]{Moosavi2019} which is direct for torus case:

\textit{Let $\Psi\in\text{Aut}(N)$ be a lift of some hyperbolic endomorphism of a nilmanifold $M = N/\Gamma$. Then for any $\varepsilon > 0$ there exists $\delta > 0$ such that $x\in\widetilde{\mathcal{L}}^s(y)$ and $d(x, y) < \delta$ implies that $\widetilde{\mathcal{L}}^u(x) \subseteq B_\varepsilon(\widetilde{\mathcal{L}}^u(y))$. Here $\widetilde{\mathcal{L}}^s$ and $\widetilde{\mathcal{L}}^u$ are stable and unstable foliations of $\ \Psi$ on $N$, and  $B_\varepsilon(S) = \bigcup_{x\in S}B_\varepsilon(x)$.}

However, this claim is \textit{not} true for some nilmanifold endomorphisms, see Example \ref{example: holonomy unbounded}. We will give a sufficient condition for this claim, that is $\Psi$ being $u$-ideal. A nilmanifold endomorphism is said to be $u$-ideal, if in the hyperbolic splitting of $\text{Lie}(N) = \mathfrak{n} = \mathfrak{n}^s\oplus\mathfrak{n}^u$, $\mathfrak{n}^u$ is an ideal. Equivalently, $[\mathfrak{n}^s, \mathfrak{n}^u] \subseteq\mathfrak{n}^u$. In \cite{DeWitt2021}, Dewitt introduces a condition called sorted spectrum, which implies that $[\mathfrak{n}^s, \mathfrak{n}^u] = 0$, hence avoids similar problems.

Under the assumption of being $u$-ideal, we get the following equivalence between the existence of conjugacy and unstable bundle.

\begin{proposition}\label{intro prop: special and conjugacy}
	Let $f$ be a nilmanifold Anosov map with linear part $\Psi$.
	\begin{enumerate}
		\item  If $f$ is topologically conjugate to $\Psi$, then $f$ is special.
		\item  If $f$ is special and $\Psi$ is $u$-ideal, then $f$ is topologically conjugate to $\Psi$.
	\end{enumerate}
\end{proposition}

\begin{remark}
	Note that, when $f$ is on a torus, then the $u$-ideal condition is automatically satisfied. Moreover, this condition is also satisfied when $f$ has one-dimensional stable bundle, see more details in Remark \ref{rmk: codimension one to u-ideal}.
\end{remark}

Combining results above, we get the following corollary which describes the existence of topological conjugacy by the complete characteristics in the sense of geometry: the existence of invariant unstable bundle, and also in the sense of statistics: the stable Lyapunov exponent.

For $r > 0$, let $r_*=\left\{\begin{array}{l}r - 1 + {\rm Lip}, \quad\ \ \ r\in\mathbb{N} \\ r, \quad r\notin \mathbb{N}\ {\rm or}\ r = +\infty\end{array}\right.$. 

\begin{corollary}\label{intro cor: equivalence}
	Let $f$ be a $C^{r + 1}\ (r>0)$ Anosov map on a nilmanifold $M$ with one-dimensional stable bundle and linear part $\Psi$. If $\ \Psi$ is totally non-invertible and horizontally irreducible, then the following statements are equivalent:
	\begin{enumerate}
		\item $f$ admits an invariant unstable bundle;
		\item $f$ is topologically conjugate to $\Psi$;
		\item Every periodic point of $f$ admits the same stable Lyapunov exponent;
		\item Every periodic point of $f$ admits the same stable Lyapunov exponent with $\Psi$;
		\item $f$ admits a $C^{r_*}$-smooth invariant unstable bundle.
	\end{enumerate}
	Moreover, each item implies that the conjugacy is $C^{r + 1}$-smooth along each stable leaf.
\end{corollary}

\begin{remark}\label{intro rmk: example}
	We note that such $\Psi$ satisfying the condition of Corollary \ref{intro cor: equivalence} exists. For example, consider the three-dimensional Heisenberg Lie algebra $\mathfrak{h} = \text{span}_\mathbb{R}\{X, Y, Z\}$, where
	$$X = \left(\begin{array}{ccc}0&1&0\\0&0&0\\0&0&0\end{array}\right), \ Y = \left(\begin{array}{ccc}0&0&0\\0&0&1\\0&0&0\end{array}\right), \ Z = \left(\begin{array}{ccc}0&0&1\\0&0&0\\0&0&0\end{array}\right).$$
	A Lie algebra automorphism defined by
	$$\psi(X, Y, Z) = (X, Y, Z)\left(\begin{array}{ccc}4&2&0\\2&2&0\\0&0&4\end{array}\right)$$
	uniquely decides a Lie group automorphism $\Psi$ of the three-dimensional Heisenberg Lie group $H$ preserving a lattice $\Gamma$, where
	\begin{align*}
		&H = \left\{\left(\begin{array}{ccc}1&x&z\\0&1&y\\0&0&1\end{array}\right): x, y, z\in\mathbb{R}\right\},\\
		&\Psi: \left(\begin{array}{ccc}1&x&z\\0&1&y\\0&0&1\end{array}\right) \mapsto \left(\begin{array}{ccc}1&4x + 2y&4z + 4xy + 4x^2 + 2y^2\\0&1&2x + 2y\\0&0&1\end{array}\right)\\
		&\Gamma = \left\{\left(\begin{array}{ccc}1&x&z\\0&1&y\\0&0&1\end{array}\right): x, y, z\in\mathbb{Z}\right\}.
	\end{align*}
\end{remark}

Actually, for a non-expanding Anosov map $f$ on a non-toral 3-nilmanifold, the following two conditions 
\begin{itemize}
    \item $f$ has one-dimensional stable bundle,
    \item the linear part $\Psi$ is totally non-invertible and horizontally irreducible, 
\end{itemize} 
hold automatically, see Remark \ref{rmk: 3-dim case}. Moreover, when $f$ is expanding, its invariant unstable bundle is the whole tangent bundle and the stable bundle vanishes, and it is well known that $f$ is topologically conjugate to $\Psi$ \cite{Shub1969}. Hence by Corollary \ref{intro cor: equivalence},  we get an immediate corollary for Anosov maps on non-toral 3-nilmanifold without any limitation.

\begin{corollary}\label{intro cor: equivalence of 3-dim case}
    Let $f$ be a $C^{r + 1}\ (r > 0)$ Anosov map on a non-toral 3-nilmanifold with linear part $\Psi$. Then the following statements are equivalent:
    \begin{enumerate}
        \item $f$ admits an invariant unstable bundle;
        \item $f$ is topologically conjugate to $\Psi$;
        \item Every periodic point of $f$ admits the same stable Lyapunov exponent;
        \item Every periodic point of $f$ admits the same stable Lyapunov exponent with $\Psi$;
        \item $f$ admits a $C^{r_*}$-smooth  invariant unstable bundle.
     \end{enumerate}
     Moreover, each item implies that the conjugacy is $C^{r + 1}$-smooth along each stable leaf.
\end{corollary}

Here is the organization of this paper.

In section \ref{Preliminaries}, we introduce some basic definitions and properties of Anosov maps and nilmanifolds. We also prove Proposition \ref{intro prop: special and conjugacy} (see also Theorem \ref{thm: special and conjugacy}): the relationship between being special and being conjugate to its linear part.

In section \ref{Density of preimages}, we prove Theorem \ref{intro thm: exponential density},  the exponential density of preimage set for totally non-invertible nilmanifold endomorphisms. The proof is an induction based on the result for the torus case \cite{An2023}.

In section \ref{Rigidity of conjugacy and stable Lyapunov exponents}, we prove Theorem \ref{intro thm: conjuacy implies pdc} (see also Theorem \ref{thm: conjugacy implies pdc}): the existence of  conjugacy between $f$ and $\Psi$ implies the  same periodic stable Lyapunov exponents, provided that $\Psi$ is totally non-invertible. Then without the assumption of conjugacy on $M$, we consider the conjugacy $H$ on the universal cover and show that constant stable Lyapunov exponent of periodic points guarantees the smoothness of $H$ restricted on the stable leaves, see Theorem \ref{thm: pdc implies H smooth}. This implies Corollary \ref{intro cor: rigidity of conjugacy} directly. If we assume further that $\Psi$ is horizontally irreducible, Theorem \ref{thm: pdc implies H smooth} also deduces Theorem \ref{intro thm: pdc implies conjugacy} (see also Theorem \ref{thm: pdc implies conjugacy}). Finally, we prove Corollary \ref{intro cor: equivalence}. 

\section{Preliminaries}\label{Preliminaries}

First of all, we introduce some notations in this paper.

For a Lie group $N$ and $x\in N$, $L_x$ is the left translation by $x$: $L_x(y) = xy$, and $R_x$ is the right translation by $x$: $R_x(y) = yx$. $Ad_x := L_{x^{-1}}\circ R_x$, that is, $Ad_x(y) = x^{-1}yx$.

The collection of automorphisms of a Lie group $N$ is denoted by $\text{Aut}(N)$.

The collection of automorphisms of a Lie algebra $\mathfrak{n}$ is denoted by $\text{Aut}(\mathfrak{n})$.
	
For a foliation $\mathcal{F}$ on a Riemannian manifold $M$ with smooth leaves, on each leaf there is an induced Riemannian metric, and hence an induced distance, denoted by $d_{\mathcal{F}}$.

Let $f, g$ be maps on a metric space $(X, d)$. $d(f, g) := \sup\{d(f(x), g(x)): x\in X\}$.

In a metric space $(X, d)$, for $x\in X$ and $S$, $S_1$, $S_2\subseteq X$, $d(x, S):=\inf\{d(x, y): y\in S\}$, $d(S_1, S_2):=\inf\{d(z_1, z_2): z_1\in S_1, z_2\in S_2\}$.

\subsection{Automorphisms and Endomorphisms of Nilmanifolds}

For an $s$-step nilpotent Lie group $N$, denote the lower central series of $N$ by $N = N_1 \triangleright \cdots \triangleright N_{s + 1} = \{e\}$. For an $s$-step nilpotent Lie algebra $\mathfrak{n}$, denote the lower central series of $\mathfrak{n}$ by $\mathfrak{n} = \mathfrak{n}_1 \triangleright\cdots\triangleright \mathfrak{n}_{s + 1} = \{0\}$. The relationship between the nilpotency of a Lie group and its Lie algebra is stated as follows. See more information about nilpotent Lie groups and nilmanifolds in \cite{Raghunathan1972}.

\begin{theorem}\label{thm: nilpotency}
	Let $N$ be a simply connected Lie group. Then $N$ is nilpotent if and only if $\mathfrak{n} = \text{Lie}(N)$ is nilpotent, and in both case they have the same step of nilpotency (denoted by $s$). Moreover, for $1\leq i < j \leq s + 1$, $N_i/N_j$ is a simply connected nilpotent Lie group with Lie algebra $\mathfrak{n}_i/\mathfrak{n}_j$ and the exponential map $\exp: \mathfrak{n}_i/\mathfrak{n}_j \to N_i/N_j$ is a diffeomorphism.
\end{theorem}

For a simply connected nilpotent Lie group $N$ admitting a lattice $\Gamma$, the right action of $\Gamma$ on $N$ is free, properly discontinuous and cocompact, hence the canonical projection $\pi: N \to N/\Gamma$ is a covering map and $M = N/\Gamma$ is a smooth closed manifold, called a {\it nilmanifold}.

Let $M = N/\Gamma$ be a nilmanifold, where $N$ is $s$-step nilpotent. Define $\Gamma_i := \Gamma\bigcap N_i$, then $\Gamma_i/\Gamma_j$ is a lattice of the simply connected nilpotent Lie group $N_i/N_j$, for $1\leq i < j \leq s + 1$. Therefore, $M_{i, j} := N_i/N_j\Gamma_i = (N_i/N_j)/(\Gamma_i/\Gamma_j)$ is also a nilmanifold. $M_{i, i + 1}$ is abelian and thus isomorphic to a torus. Write $M_i = M_{1, i + 1}$, $1\leq i \leq s$, then $M = M_s$. The torus $M_1 = N/N_2\Gamma$ is called the {\it horizontal torus} of $M = N/\Gamma$.

A left {\it principal $G$-bundle} is a fiber bundle $(E, B, F, \pi)$ equipped with a continuous free left $G$-action on $E$ that preserves and acts transitively on every fiber. Here $G$ is a topological group, $E$, $B$, $F$, $\pi$ are the total space, the base space, the typical fiber and the canonical projection of the fiber bundle respectively. Every fiber of a left principal $G$-bundle is homeomorphic to $G$.

\begin{theorem}\label{thm: torus bundle over torus}
{\rm(\cite{Raghunathan1972}\cite{Palais1961})} Let $M = N/\Gamma$ be a nilmanifold. Then $M_i$ is a left principal $M_{i, i + 1}$-bundle over $M_{i - 1}$, and the fiber $M_{i, i + 1} = \mathbb{T}^{d_i}$, where $d_i = \dim\mathfrak{n}_i - \dim\mathfrak{n}_{i + 1}$.
\end{theorem}

An {\it automorphism} of a nilmanifold $M = N/\Gamma$, is that induced by an automorphism $\Psi\in \text{Aut}(N)$ satisfying $\Psi(\Gamma) = \Gamma$. An {\it endomorphism} of $M$ is that induced by an automorphism $\Psi\in \text{Aut}(N)$ satisfying $\Psi(\Gamma)\subseteq\Gamma$. The collection of automorphisms and endomorphisms of $M$ is denoted by $\text{Aut}(M)$ and $\text{End}(M)$ respectively. Clearly, nilmanifold automorphisms are diffeomorphisms, and nilmanifold endomorphisms are local diffeomorphisms.

$\text{Aut}(N)$ is identified with $\text{Aut}(\mathfrak{n})$ via $\Psi\mapsto \psi = D_e\Psi$. $\text{Aut}(M) = \{\Psi\in\text{Aut}(N): \Psi(\Gamma) = \Gamma\}$ is identified with $\text{Aut}(\Gamma)$. $\text{End}(M) = \{\Psi\in\text{Aut}(N): \Psi(\Gamma)\subseteq\Gamma\}$ is identified with the collection of monomorphisms of $\Gamma$. The main idea is that a monomorphism of $\Gamma$ can be uniquely extended to an automorphism of $N$, see \cite{Dekimpe2012}. In this paper, the endomorphism induced by some $\Psi\in\text{Aut}(N)$ satisfying $\Psi(\Gamma)\subseteq\Gamma$ is also denoted by $\Psi\in\text{End}(M)$ and we do not distinguish them unless necessary. Moreover, the eigenvalues of $\psi = D_e\Psi\in\text{Aut}(\mathfrak{n})$ is also called the eigenvalues of $\Psi$.

For a simply connected $s$-step nilpotent Lie group $N$ and its Lie algebra $\mathfrak{n}$, a {\it Mal'cev basis} of $\mathfrak{n}$ adapted to the lower central series, is a basis $\{X_1, \cdots, X_d\}$, such that $\{X_{d - \dim\mathfrak{n}_i + 1}, \cdots, X_d\}$ is a basis of $\mathfrak{n}_i$, $1\leq i \leq s$. Under such a basis, the Lie bracket is uniquely decided by $[X_i, X_j] = c_{ij}^kX_k$, the constants $\{c_{ij}^k\}$ are called {\it the structural constants} with respect to the basis. $N$ admits a lattice $\Gamma$, if and only if $\mathfrak{n}$ admits a Mal'cev basis with rational structural constants. In this case, the Mal'cev basis can be properly chosen such that $\Gamma = \{(\exp n_1X_1)\cdots(\exp n_dX_d): n_1, \cdots, n_d\in\mathbb{Z}\}$.

For $\Psi\in\text{End}(M)$ and $\psi = D_e\Psi\in\text{Aut}(\mathfrak{n})$, take a Mal'cev basis adapted to the lower central series, then $\psi$ has a block matrix representation
$$\psi(\textbf{X}_1, \cdots, \textbf{X}_s) = (\textbf{X}_1, \cdots, \textbf{X}_s)\left(\begin{matrix}\psi_1 & & &\\ * & \psi_2 & &\\ \vdots & \vdots & \ddots & \\ * & * & \cdots &\psi_s\end{matrix}\right),$$
where $\textbf{X}_i = (X_{d - \dim\mathfrak{n}_i + 1}, \cdots, X_{d - \dim\mathfrak{n}_{i + 1}})$. Actually $\textbf{X}_i$ is projected to a basis of $\mathfrak{n}_i/\mathfrak{n}_{i + 1}$.

Clearly $\psi_i\in\text{Aut}(\mathfrak{n}_i/\mathfrak{n}_{i + 1})$. Notice that $\mathfrak{n}_i/\mathfrak{n}_{i + 1}$ is abelian, so $\mathfrak{n}_i/\mathfrak{n}_{i + 1}$ is identified with $N_i/N_{i + 1}$, both of which is identified with $\mathbb{R}^{d_i}$. Moreover, $\Gamma_i/\Gamma_{i + 1}$ is identified with $\mathbb{Z}^{d_i}$, and $\psi_i$ is identified with $\Psi_{i, i + 1}\in\text{Aut}(N_i/N_{i + 1})$ induced by $\Psi$. On the other hand, $\Psi(\Gamma)\subseteq\Gamma$, so $\Psi_{i, i + 1}(\Gamma_i/\Gamma_{i + 1})\subseteq\Gamma_i/\Gamma_{i + 1}$ and thus $\psi_i\in GL(d_i, \mathbb{R})\bigcap M(d_i, \mathbb{Z})$. When $\Psi\in\text{Aut}(M)$, we have $\Psi_{i, i + 1}(\Gamma_i/\Gamma_{i + 1}) = \Gamma_i/\Gamma_{i + 1}$ and thus $\psi_i\in GL(d_i, \mathbb{Z})$. As a corollary, the eigenvalues of $\Psi\in\text{End}(M)$ are algebraic integers (roots of monic polynomials with $\mathbb{Z}$-coefficients), and the eigenvalues of $\Psi\in\text{Aut}(M)$ are algebraic units (roots of monic polynomials with $\mathbb{Z}$-coefficients and constant term $\pm 1$).

Since $N/N_2$ is abelian, it is identified with its Lie algebra $\mathfrak{n}/\mathfrak{n}_2$ by the exponential map. Any endomorphism $\Psi\in\text{End}(M)$ induces a toral endomorphism $\Psi_1 \in\text{End}(M_1)$ naturally, which is called the {\it horizontal part} of $\Psi$. Of course the induced automorphism on $N/N_2$ by $\Psi\in\text{Aut}(N)$ is also denoted by $\Psi_1\in\text{Aut}(N/N_2)$. Since $N/N_2$ is identified with $\mathfrak{n}/\mathfrak{n}_2$, $\Psi_1\in\text{Aut}(N/N_2)$ is also identified with $\psi_1 \in \text{Aut}(\mathfrak{n}/\mathfrak{n}_2)$, which is induced by $\psi = D_e\Psi\in\text{Aut}(\mathfrak{n})$.

The following lemma shows the importance of horizontal part.

\begin{lemma}\label{lem: horizontal part}
	The following statements hold.
	
	(1)\ $\psi_2, \cdots, \psi_s$ are determined by $\psi_1$.
	
	(2)\ Every eigenvalue of $\ \psi_i$ is the product of $\ i$ eigenvalues of $\ \psi_1$.
\end{lemma}

\begin{proof}
	(1) Let $\pi_i: \mathfrak{n}_i \to\mathfrak{n}_i/\mathfrak{n}_{i + 1}$ be the natural projection. Notice that $\psi_i$ is defined by $\psi_i\circ\pi_i = \pi_i\circ\psi|_{\mathfrak{n}_i}: \mathfrak{n}_i \to \mathfrak{n}_i/\mathfrak{n}_{i + 1}$, and $\mathfrak{n}_i$ is spanned by vectors in the form of $L(Y_1, \cdots, Y_i) = [Y_1, \cdots [Y_{i - 1}, Y_i]\cdots]$, so it suffices to show that $\pi_i \circ L(\psi Y_1, \cdots, \psi Y_i)$ is determined by $\psi_1: \mathfrak{n}/\mathfrak{n}_2 \to \mathfrak{n}/\mathfrak{n}_2$ and $(Y_1, \cdots, Y_i)$.
	
	When $Y_j \in\mathfrak{n}_2$ for some $1\leq j \leq i$, one has $L(Y_1, \cdots, Y_i) \in \mathfrak{n}_{i + 1}$. Therefore $\pi_i\circ L$ is well-defined on $\Pi_{j = 1}^i(\mathfrak{n}/\mathfrak{n}_2)$, and  $\pi_i \circ L(\psi Y_1, \cdots, \psi Y_i) = \pi_i\circ L(\psi_1(Y_1\mathfrak{n}_2), \cdots, \psi_1(Y_i\mathfrak{n}_2))$, since $\psi$ preserves $\mathfrak{n}_2$.
	
	(2) It suffices to show that every eigenvalue of $\ \psi_{i + 1}$ is the product of an eigenvalue of $\ \psi_1$ and an eigenvalue of $\psi_i$. Let $\Lambda_i$ be the eigenvalues of $\psi_i$. $E_i^\lambda = \{v\in\mathfrak{n}_i^\mathbb{C}: (\lambda I - \psi)^nv\in\mathfrak{n}_{i + 1}^\mathbb{C} \text{ for sufficiently large } n\}$, then $\mathfrak{n}_i^\mathbb{C} = \sum_{\lambda\in\Lambda_i}E_i^\lambda$, and $\mathfrak{n}_i^\mathbb{C}/\mathfrak{n}_{i + 1}^\mathbb{C} = \bigoplus_{\lambda\in\Lambda_i}(E_i^\lambda\mathfrak{n}_{i + 1}^\mathbb{C})$.
	
	Claim that $[E_1^\lambda , E_i^\mu] \subseteq E_{i + 1}^{\lambda\mu}$. In fact, take $X\in E_1^\lambda$ and $Y\in E_i^\mu$, then
	$$(\psi - \lambda\mu I)[X, Y] = [\psi X, \psi Y] - [\lambda X, \mu Y] = [(\psi - \lambda I)X, \psi Y] + [\lambda X, (\psi - \mu I)Y],$$
	hence
	$$(\psi - \lambda\mu I)^n[X, Y] = \sum_{k = 0}^n C_n^k [\lambda^k(\psi - \lambda I)^{n - k}X, \psi^{n - k}(\psi - \mu I)^kY].$$
	When $n$ is sufficiently large, every term in the sum lies in $[\mathfrak{n}_2^\mathbb{C}, \mathfrak{n}_i^\mathbb{C}]$ or $[\mathfrak{n}^\mathbb{C}, \mathfrak{n}_{i + 1}^\mathbb{C}]$. Both of them lie in $\mathfrak{n}_{i + 2}^\mathbb{C}$, thus $[X, Y]\in E_{i + 1}^{\lambda\mu}$.
	
	Now by the claim, we have
	$$\mathfrak{n}_{i + 1}^\mathbb{C} = [\mathfrak{n}_1^\mathbb{C}, \mathfrak{n}_i^\mathbb{C}] = \left[\sum_{\lambda\in\Lambda_1}E_1^\lambda, \sum_{\mu\in\Lambda_i}E_i^\mu\right] \subseteq \sum_{\lambda\in\Lambda_1}\sum_{\mu\in\Lambda_i}E_{i + 1}^{\lambda\mu},$$
	and the proof is completed.
\end{proof}

Nilmanifold endomorphisms are covering maps. General covering maps on nilmanifolds are related to endomorphisms, see the following discussion.

Let $M = N/\Gamma$ be a nilmanifold. The canonical projection $\pi: N \to M$ is the universal cover of $M$. Choose some $x_0\in M$ as the base point of $M$ and some $\widetilde{x}_0\in \pi^{-1}(x_0)$ as the base point of $N$, then there is an isomorphism between $\pi_1(M, x_0)$ and $\Gamma$ given by $[r] \mapsto \widetilde{r}(0)^{-1}\widetilde{r}(1)$, where $r: [0, 1] \to M$ is a loop at $x_0$, $\widetilde{r}: [0, 1] \to N$ is the unique lift of $r$ satisfying $\widetilde{r}(0) = \widetilde{x}_0$.

Further, a covering map $f: M \to M$ induces a monomorphism $f_*: \pi_1(M, x_0) \to \pi_1(M, y_0)$, where $y_0 = f(x_0)$. Choose $\widetilde{x}_0\in\pi^{-1}(x_0)$ and $\widetilde{y}_0\in\pi^{-1}(y_0)$ respectively, then there is a unique lift of $f$, denoted by $F$, satisfying $F(\widetilde{x}_0) = \widetilde{y}_0$. It appears that there exists a unique monomorphism $\Psi: \Gamma \to \Gamma$ such that $F(n\gamma) = F(n)\Psi(\gamma)$, $\forall n\in N$, $\gamma\in\Gamma$. Such $\Psi$ is uniquely extended to $\Psi\in\text{Aut}(N)$ and is called the {\it linear part} of $F$. Moreover, $\Psi = f_*$ when identify $\pi_1(M, x_0)$ and $\pi_1(M, y_0)$ with $\Gamma$ respectively.

Note that $\Psi$ depends on $F$, whereas $F$ depends on the choice of base points. Consider another lift $F'$, we have $F'(n) = F(n)\gamma_0$ for some $\gamma_0\in\Gamma$, and hence
$$F'(n\gamma) = F(n\gamma)\gamma_0 = F(n)\Psi(\gamma)\gamma_0 = F'(n)\gamma_0^{-1}\Psi(\gamma)\gamma_0,$$
which means $\Psi'(\gamma) = \gamma_0^{-1}\Psi(\gamma)\gamma_0$, that is, $\Psi' = Ad_{\gamma_0}\circ \Psi$.

Since the monomorphisms of $\Gamma$ is identified with the endomorphisms of $M$, we also write $\Psi\in\text{End}(M)$. It follows that $\Psi'(x\Gamma) = \gamma_0^{-1}\Psi(x\Gamma)$, i.e., $\Psi' = L_{\gamma_0^{-1}}\circ\Psi \in \text{End}(M)$. It follows that $\Psi'\in\text{End}(M)$ and $\Psi\in\text{End}(M)$ are homotopic. In this sense, $\Psi\in\text{End}(M)$ is called the {\it linear part} of $f$.

Actually $f$ is homotopic to its linear part. To see this, consider a homotopy between $F$ and $\Psi$,
$$H(t, n) = (\exp t(\exp^{-1}(F(n)\Psi(n)^{-1})))\Psi(n).$$
Since $H(t, n\gamma) = H(t, n)\Psi(\gamma$), $\forall \gamma\in\Gamma$, $H$ is projected to a homotopy between $f$ and $\Psi$.

\subsection{Anosov maps and hyperbolic endomorphisms}

\begin{definition}\label{def: Anosov diffeo}
    A diffeomorphism $f$ on a Riemannian manifold $M$ is {\it Anosov}, if there are constants $C > 1$, $0 < \lambda < 1$, and a $Df$-invariant splitting $TM = E^s\oplus E^u$, such that $$\left\|Df^n|_{E^s(x)}\right\|\leq C\lambda^n \ \text{and}\  \left\|Df^{-n}|_{E^u(x)}\right\|\leq C\lambda^n, \ \forall x\in M, \ \forall n \geq 0.$$
	Such a splitting is actually unique and continuous, and is called the hyperbolic splitting.
\end{definition}

\begin{definition}\label{def: Anosov map}
	A local diffeomorphism $f$ on a closed manifold $M$ is {\it Anosov}, if a lift $F$, which is a diffeomorphism of $\widetilde{M}$, is Anosov. Here $\pi: \widetilde{M} \to M$ is the universal cover of $M$.
\end{definition}

\begin{remark}
	Definition \ref{def: Anosov map} does not depend on the choice of Riemannian metrics on $M$ and lifts of $f$. Besides, local diffeomorphisms on closed manifolds are always covering maps.
\end{remark}

When $f$ is actually a diffeomorphism, the definition of an Anosov map coincides with an Anosov diffeomorphism, because in this case a hyperbolic splitting of $TM$ is pulled back to a hyperbolic splitting of $T\widetilde{M}$, and a hyperbolic splitting of $T\widetilde{M}$ is projected to a hyperbolic splitting of $TM$.

\begin{definition}\label{def: special}
	An Anosov map $f$ on a closed manifold $M$ is {\it special}, if the hyperbolic splitting of $T\widetilde{M}$ is projected to a hyperbolic splitting of $TM$.
\end{definition}

An Anosov map $f$ on a closed manifold $M$ is special if and only if the hyperbolic splitting $T\widetilde{M} = \widetilde{E}^s\oplus\widetilde{E}^u$ of $F$ is invariant under deck transformations. Note that $\widetilde{E}^s$ is always invariant under deck transformations, hence we only need the condition for $\widetilde{E}^u$. When $f$ is an Anosov diffeomorphism or an expanding map (which means $T\widetilde{M} = \widetilde{E}^u$), the condition holds and hence $f$ is special.

Now we consider Anosov maps on nilmanifolds. In the following discussion of this subsection, $M = N/\Gamma$ is a nilmanifold.

\begin{definition}\label{def: hyperbolic}
	$\Psi\in\text{End}(M)$ is {\it hyperbolic}, if the eigenvalues of $\Psi$ are not of modulus one.
\end{definition}

Clearly, $\Psi\in\text{End}(M)$ is hyperbolic if and only if $\psi = D_e\Psi\in\text{Aut}(\mathfrak{n})$ is hyperbolic. In this case, we have a hyperbolic splitting $\mathfrak{n} = \mathfrak{n}^s\oplus\mathfrak{n}^u$ as subspaces. Generally $\mathfrak{n}^s$ and $\mathfrak{n}^u$ are subalgebras, but $[\mathfrak{n}^s, \mathfrak{n}^u]$ may not vanish, i.e., $\mathfrak{n}^s$ and $\mathfrak{n}^u$ may not be ideals.

\begin{remark}\label{rmk: codimension one to u-ideal}
	For a hyperbolic endomorphism $\Psi$ with one-dimensional stable bundle, by Lemma \ref{lem: horizontal part}, the unique eigenvalue with modulus smaller than 1 must be an eigenvalue of $\psi_1$, and hence $[\mathfrak{n}^s, \mathfrak{n}^u] \subseteq \mathfrak{n}_2 \subseteq \mathfrak{n}^u$, $\mathfrak{n}^u$ is an ideal. On the other hand, if $\Psi$ has one-dimensional unstable bundle, since each of $\psi_1$, $\cdots$, $\psi_s$ has at least one eigenvalue with modulus bigger than 1, this forces $N_2$ to vanish and hence $\Psi$ is a toral endomorphism.
\end{remark}

\begin{remark}\label{rmk: 3-dim case}
	If $M$ is a non-toral 3-nilmanifold, then a non-expanding hyperbolic endomorphism $\Psi\in\text{End}(M)$ must have one-dimensional stable bundle. Moreover, the induced endomorphism $\Psi_1 \in \text{End}(N/N_2\Gamma)$ is irreducible, since $\Psi_1$ is a non-expanding endomorphism of $\mathbb{T}^2$. Let $\{X, Y, Z\}$ be the related Mal'cev basis of $\mathfrak{n} = \text{Lie}(N)$, then $\text{span}_{\mathbb{R}}\{Z\} = \text{span}_{\mathbb{R}}\{[X, Y]\}$, hence $|\det \psi_1| = |\psi_2| > 1$. Consequently, the eigenvalues of $\Psi$ are not algebraic units, and $\Psi$ is totally non-invertible, see Definition \ref{def: totally non-invertible}.
\end{remark}

\begin{lemma}\label{hyperbolic linear part}
	{\rm(\cite[Lemma 1.3]{Sumi1996})}The linear part of a nilmanifold Anosov map is hyperbolic.
\end{lemma}

Hyperbolic endomorphisms of nilmanifolds are special Anosov maps. To see this, we take a right-invariant Riemannian metric on $N$, so that it is projected to a Riemannian metric on $M = N/\Gamma$. Then the hyperbolic splitting $\mathfrak{n} = \mathfrak{n}^s\oplus\mathfrak{n}^u$ induces a hyperbolic splitting $TN = \widetilde{L}^s\oplus\widetilde{L}^u$ by right translation. The splitting is $D\Psi$-invariant and right-invariant, thus is projected to a hyperbolic splitting $TM = L^s\oplus L^u$, which is $D\Psi$-invariant.

Further, $\mathfrak{n}^\sigma$ is a Lie subalgebra of $\mathfrak{n}$, therefore $\exp\mathfrak{n}^\sigma$ is a simply connected closed Lie subgroup and decides a $\Psi$-invariant, right-invariant smooth foliation $\widetilde{\mathcal{L}}^\sigma$ on $N$ by $\widetilde{\mathcal{L}}^\sigma(n) = (\exp\mathfrak{n}^\sigma)n$, which is projected to a $\Psi$-invariant smooth foliation $\mathcal{L}^\sigma$ on $M$, $\sigma = s$, $u$. In fact, $\widetilde{\mathcal{L}}^s$ and $\widetilde{\mathcal{L}}^u$ are stable and unstable foliations of $\Psi \in \text{Aut}(N)$, $\mathcal{L}^s$ and $\mathcal{L}^u$ are stable and unstable foliations of $\Psi\in \text{End}(M)$.

Generally, for an Anosov map $f$ on a nilmanifold $M = N/\Gamma$, let $F$ be a lift of $f$. Since $F$ is an Anosov diffeomorphism, there is a hyperbolic splitting $TN = \widetilde{E}^s\oplus \widetilde{E}^u$. Moreover, there are stable foliations $\widetilde{\mathcal{F}}^s$ and unstable foliations $\widetilde{\mathcal{F}}^u$. If $f$ is special, then the splitting is projected to a hyperbolic splitting $TM = E^s\oplus E^u$, and the stable and unstable foliations are also projected to foliations on $M$, denoted by $\mathcal{F}^s$ and $\mathcal{F}^u$.

Note that $\widetilde{\mathcal{F}}^\sigma$ is $F$-invariant, $\widetilde{\mathcal{F}}^s$ is right-$\Gamma$-invariant; $\widetilde{\mathcal{F}}^u$ is right-$\Gamma$-invariant if and only if $f$ is special. Moreover, $\widetilde{\mathcal{L}}^\sigma$ is $\Psi$-invariant and right-invariant, $\sigma = s$, $u$.

\begin{lemma}\label{lem: global product structure}
    {\rm (\cite[Lemma 7.6]{Sumi1996})} For any $x, y\in N$, $\widetilde{\mathcal{F}}^s(x)\bigcap \widetilde{\mathcal{F}}^u(y)$ is the set of one point.
\end{lemma}

Denote the unique point in $\widetilde{\mathcal{F}}^s(x)\bigcap \widetilde{\mathcal{F}}^u(y)$ by $\beta_{\widetilde{\mathcal{F}}}(x, y)$. In particular, for the algebraic case we denote $\beta = \beta_{\widetilde{\mathcal{L}}}$ for simplicity, which does not cause confusion in this paper. Immediately, we have $\beta(xz, yz) = \beta(x, y)z$, $\Psi\beta(x, y) = \beta(\Psi(x), \Psi(y))$.

\begin{proposition}\label{prop: su decomposition}
	For any $x\in N$, there is a unique decomposition $x = x^sx^u$, $x^s\in\widetilde{\mathcal{L}}^s(e)$, $x^u\in\widetilde{\mathcal{L}}^u(e)$. $x\mapsto x^s$ and $x\mapsto x^u$ are both smooth.
\end{proposition}

\begin{proof}
	We will show that $P: \widetilde{\mathcal{L}}^s(e)\times\widetilde{\mathcal{L}}^u(e) \to N$, $P(y, z) = yz$, is a diffeomorphism.
	
	Prove inductively that $P_i = P|_{\widetilde{\mathcal{L}}_i^s(e)\times \widetilde{\mathcal{L}}_i^u(e)}: \widetilde{\mathcal{L}}_i^s(e)\times \widetilde{\mathcal{L}}_i^u(e) \to N_i$ is a diffeomorphism, $1\leq i \leq s + 1$, where $\widetilde{\mathcal{L}}_i^\sigma(e) = \widetilde{\mathcal{L}}^\sigma(e) \bigcap N_i$ is a simply connected closed Lie subgroup of $N_i$, with Lie algebra $\mathfrak{n}_i^\sigma = \mathfrak{n}^\sigma\bigcap\mathfrak{n}_i$, $\sigma = s$, $u$, satisfying $\mathfrak{n}_i^s\oplus \mathfrak{n}_i^u = \mathfrak{n}_i$.
	
	\begin{claim*}
		$P_i$ is a local diffeomorphism.
	\end{claim*}
	
	\begin{proof}[Proof of Claim]
        Notice that for any $y\in\widetilde{\mathcal{L}}_i^s(e)$, $z\in\widetilde{\mathcal{L}}_i^u(e)$, $Y\in T_y\widetilde{\mathcal{L}}_i^s(e)$, $Z\in T_z\widetilde{\mathcal{L}}_i^u(e)$, we have
		$$D_{(y, z)}P_i(Y, Z) = D_yR_zY + D_zL_yZ.$$
		Since $T_y\widetilde{\mathcal{L}}_i^s(e) = D_eR_y\mathfrak{n}_i^s$, $T_z\widetilde{\mathcal{L}}_i^u(e) = D_eR_z\mathfrak{n}_i^u$, it suffices to show that $$D_yR_z\circ D_eR_y \oplus D_zL_y\circ D_eR_z: \mathfrak{n}_i^s \oplus \mathfrak{n}_i^u \to T_{yz}N_i$$
		is an isomorphism, or equivalently, an injection. Notice that $D_zL_y\circ D_eR_z = D_yR_z\circ D_eL_y$, the question reduces to whether $D_eR_y \oplus D_eL_y: \mathfrak{n}_i^s\oplus \mathfrak{n}_i^u \to T_yN_i$ is an injection. Assume that there exists $Y\in \mathfrak{n}_i^s$, $Z\in \mathfrak{n}_i^u$ such that $D_eR_yY = D_eL_yZ$. Then $y^{-1}\exp(tY)y$ is tangent to $D_yL_{y^{-1}}\circ D_eR_y Y= Z$. But this curve lies in $\widetilde{\mathcal{L}}_i^s(e)$, which forces $Y$ and $Z$ to vanish.
	\end{proof}

    It follows that $\text{Im} P_i$ consists some neighborhood of $e\in N_i$. Besides, $P_i$ is injective, because $yz = y'z'$ for some $y$, $y'\in \widetilde{\mathcal{L}}_i^s(e)$ and $z$, $z'\in\widetilde{\mathcal{L}}_i^u(e)$ implies $(y')^{-1}y = z'z^{-1} \in \widetilde{\mathcal{L}}_i^s(e)\bigcap \widetilde{\mathcal{L}}_i^u(e) = \{e\}$ and thus $y' = y$, $z' = z$.

    Now one only needs to prove inductively that $P_i$ is surjective. The case when $i = s + 1$ is obvious. Assume that $\text{Im} P_{i + 1} \supseteq N_{i + 1}$. To show that $\text{Im}P_i \supseteq N_i$, it suffices to show that $\text{Im} P_i$ is closed under multiplication.

    Take $y, y'\in \widetilde{\mathcal{L}}_i^s(e)$, $z, z' \in \widetilde{\mathcal{L}}_i^u(e)$. One needs to show that $(yz)(y'z')\in\text{Im} P_i$. In fact, $(yz)(y'z') = yy'[(y')^{-1}, z]zz'$. $[(y')^{-1}, z] \in N_{i + 1} \subset \text{Im} P_{i + 1}$, thus there exists $y''\in \widetilde{\mathcal{L}}_{i + 1}^s(e)$, $z''\in\widetilde{\mathcal{L}}_{i + 1}^u(e)$ such that $[z^{-1}, y'] = y''z''$. Thus $(yz)(y'z') = (yy'y'')(z''zz') \in \text{Im} P_i$ and the induction is completed.
\end{proof}

\subsection{Conjugacy with linear part}
\begin{lemma}\label{lem: affine conjugate auto}
	{\rm (\cite[ Lemma 1.4]{Sumi1996})} Assume that $\Psi\in\text{Aut}(N)$ is hyperbolic, then $\xi(x) = x^{-1}\Psi(x)$ and $\eta(x) = \Psi(x)x^{-1}$ are both diffeomorphisms on $N$.
\end{lemma}

\begin{corollary}\label{cor: affine conjugate auto}
	Assume that $\Psi\in\text{Aut}(N)$ is hyperbolic, then for every $y\in N$, $T = L_y\circ\Psi$ is conjugate to $\Psi$ via $L_x$ for some $x\in N$, i.e., $L_x\circ T = \Psi\circ L_x$.
\end{corollary}

\begin{proof}
	$L_x\circ T = \Psi\circ L_x$ is equivalent to $L_x\circ L_y\circ\Psi = L_{\Psi(x)}\circ\Psi$, and also $y = x^{-1}\Psi(x)$. By Lemma \ref{lem: affine conjugate auto}, such $x$ uniquely exists.
\end{proof}

Recall that, for an Anosov map $f$ on a nilmanifold $M = N/\Gamma$, different lifts $F$ and $F'$ have different linear parts $\Psi\in\text{Aut}(N)$ and $\Psi'\in\text{Aut}(N)$, but the induced endomorphisms $\Psi\in\text{End}(M)$ and $\Psi'\in\text{End}(M)$ satisfy $\Psi' = L_{\gamma_0^{-1}}\circ \Psi$ for some $\gamma_0\in\Gamma$. They are not only homotopic, but also algebraically conjugate to each other. Therefore, there is no confusion when saying an Anosov map is conjugate to its linear part.

In the following discussion of this subsection, $f$ is an Anosov map on a nilmanifold $M = N/\Gamma$, $F$ is a lift of $f$, $\Psi$ is the linear part of $F$.

\begin{lemma}\label{lem: fixed point of F}
	{\rm (\cite[ Lemma 1.5]{Sumi1996})} $F$ has a unique fixed point.
\end{lemma}

\begin{lemma}\label{lem: conjugate}
	{\rm (\cite[Lemma 2.3, Lemma 7.13]{Sumi1996})} There is a unique map $H: N \to N$ satisfying the following properties.
	\begin{itemize}
		\item $\Psi\circ H = H \circ F$;
		\item $d(H, \text{\rm Id}_N) < +\infty$.
	\end{itemize}
	Moreover, $H$ is a bi-uniformly continuous homeomorphism and $d(H^{-1}, \text{\rm Id}_N) < +\infty$.
\end{lemma}

We always write $C_0 = \max\{d(H, \text{\rm Id}_N), d(H^{-1}, \text{\rm Id}_N), 1\}$ in this paper. Besides, let $b\in N$ be the unique fixed point of $F$, consider $\widehat{f} = L_b^{-1}\circ f\circ L_b$ and its lift $\widehat{F} = L_b^{-1}\circ F\circ L_b$. They have the same linear part with $f$ and $F$. Moreover, $\widehat{F}(e) = e$. Thus we may assume at first that $b = e$. Consequently, $H(e) = e$.

We also note that $H$ is a leaf conjugacy, which means $H(\widetilde{\mathcal{F}}^\sigma(x)) = \widetilde{\mathcal{L}}^\sigma(H(x))$, $\sigma = s, u$.

A natural question is whether $H$ is projected to a homeomorphism of $M$, so that $f$ is topologically conjugate to $\Psi$. By the results of \cite{Sumi1996} for nilmanifolds, the answer is yes when $f$ is expanding. \cite{Moosavi2019} deals with the case when $f$ is special and non-expanding. We shall discuss the question in subsection \ref{special Anosov}. Before that we need some properties of $H$.

Let $d$ be the distance induced by the right-invariant Riemannian metric on $N$. Clearly $d$ is also right-invariant, i.e., $d(xz, yz) = d(x, y)$, $\forall x$, $y$, $z\in N$.

\begin{lemma} \label{lem: H and Gamma}
	Let $H$ be as in Lemma \ref{lem: conjugate}.
	\begin{enumerate}
		\item $H(x\gamma)\gamma^{-1} \in \widetilde{\mathcal{L}}^s(H(x))$;
		\item $H^{-1}(x\gamma)\gamma^{-1} \in \widetilde{\mathcal{F}}^s(H^{-1}(x))$.
	\end{enumerate}
\end{lemma}

\begin{proof}
	We claim that $y\in\widetilde{\mathcal{L}}^s(x)$ if and only if $\sup_{i\geq 0}d(\Psi^i(x), \Psi^i(y)) < +\infty$. In fact, the necessity is obvious, and to show sufficiency, assume for contradiction that $\sup_{i\geq 0}d(\Psi^i(x), \Psi^i(y)) < +\infty$ and $y\not\in\widetilde{\mathcal{L}}^s(x)$. Take $z = \beta(x, y)$, then $z\not= y$, thus $d(\Psi^i(z), \Psi^i(y)) \to +\infty (i \to +\infty)$, and hence $d(\Psi^i(z), \Psi^i(x)) \to +\infty (i \to +\infty)$, which contradicts with $z\in\widetilde{\mathcal{L}}^s(x)$.
	
	For 1, recall that $d(H, \text{\rm Id}_N) \leq C_0$. We have
	\begin{align*}
		d(\Psi^i(H(x\gamma)\gamma^{-1}), \Psi^i(H(x)))
		&= d(H(F^i(x\gamma))\Psi^i(\gamma^{-1}), H(F^i(x)))\\
		&\leq C_0 + d(F^i(x\gamma), H(F^i(x))\Psi^i(\gamma))\\
		&= C_0 + d(F^i(x)\Psi^i(\gamma), H(F^i(x))\Psi^i(\gamma))\\
		&\leq C_0 + d(F^i(x), H(F^i(x))) \leq 2C_0.
	\end{align*}
	Therefore, $H(x\gamma)\gamma^{-1} \in \widetilde{\mathcal{L}}^s(H(x))$.
	
	For 2, from 1 we have $H(x\gamma) \in \widetilde{\mathcal{L}}^s(H(x)\gamma)$, and hence $x\gamma\in H^{-1}(\widetilde{\mathcal{L}}^s(H(x)\gamma)) = \widetilde{\mathcal{F}}^s(H^{-1}(H(x)\gamma))$, or equivalently, $H^{-1}(H(x)\gamma) \in \widetilde{\mathcal{F}}^s(x\gamma)$. Replace $x$ by $H^{-1}(x)$ and we have $H^{-1}(x\gamma) \in \widetilde{\mathcal{F}}^s(H^{-1}(x)\gamma)$, that is, $H^{-1}(x\gamma)\gamma^{-1} \in \widetilde{\mathcal{F}}^s(H^{-1}(x))$, since $\widetilde{\mathcal{F}}^s$ is right-$\Gamma$-invariant.
\end{proof}

We note that if $f$ is expanding, then by Lemma \ref{lem: H and Gamma}, $H$ commutes with $\Gamma$ and thus is projected to a conjugacy between $f$ and $\Psi$.

\begin{lemma}\label{lem: sequence}
	Let $H$ be as in Lemma \ref{lem: conjugate}. There exist positive constants $\varepsilon_k \to 0$ as $k \to+\infty$ such that for any sequence $\gamma_k\in \Psi^k\Gamma$, $k \geq 1$, the followings hold:
	\begin{enumerate}
		\item $d(H(x\gamma_k), H(x)\gamma_k) \leq 2C_0\mu^s_+(\Psi)^k$, $\forall x\in N$, $\forall k \geq 1$;
		\item $d(H^{-1}(x\gamma_k), H^{-1}(x)\gamma_k) \leq \varepsilon_k$, $\forall x\in N$, $\forall k \geq 1$.
	\end{enumerate}
	Here $0 < \mu^s_+(\Psi) < 1$ is the maximal modulus of eigenvalues of $\ \Psi$ with modulus smaller than 1.
\end{lemma}

\begin{proof}
	$\gamma_k\in\Psi^k\Gamma$ implies that $\gamma'_k := \Psi^{-k}(\gamma_k)\in\Gamma$. Therefore, $F^{-k}(x\gamma_k) = F^{-k}(x)\gamma'_k$, and
	\begin{align*}
		d(H(F^{-k}(x)\gamma'_k), H(F^{-k}(x))\gamma'_k)
		&\leq C_0 + d(F^{-k}(x)\gamma'_k, H(F^{-k}(x))\gamma'_k)\\
		& = C_0 + d(F^{-k}(x), H(F^{-k}(x))) \leq 2C_0,
	\end{align*}
	or equivalently, $d(H(F^{-k}(x)\gamma'_k)(\gamma'_k)^{-1}, H(F^{-k}(x))) \leq 2C_0$.
	
	By Lemma \ref{lem: H and Gamma}, $H(F^{-k}(x)\gamma'_k)(\gamma'_k)^{-1}\in\widetilde{\mathcal{L}}^s(H(F^{-k}(x)))$, so
	$$d(\Psi^k(H(F^{-k}(x)\gamma'_k)(\gamma'_k)^{-1}), \Psi^k(H(F^{-k}(x)))) \leq 2C_0\mu^s_+(\Psi)^k.$$
	Equivalently, $d(H(x\gamma_k)\gamma_k^{-1}, H(x))\leq 2C_0\mu^s_+(\Psi)^k$, this proves the first property.
	
	For the second property, by uniform continuity of $H^{-1}$, there are positive constants $\varepsilon_k \to 0$ as $k \to +\infty$, such that $d(x, y) \leq 2C_0\mu^s_+(\Psi)^k$ implies that $d(H^{-1}(x), H^{-1}(y)) \leq\varepsilon_k$. By the first property, we have $d(x\gamma_k, H^{-1}(H(x)\gamma_k)) \leq\varepsilon_k$. Replace $x$ with $H^{-1}(x)$, then $d(H^{-1}(x)\gamma_k, H^{-1}(x\gamma_k)) \leq\varepsilon_k$.
\end{proof}

We note that if $f$ is actually a diffeomorphism, then $\Psi\in\text{Aut}(\Gamma)$ and hence $\Psi^k\Gamma = \Gamma$, $\forall k \geq 1$. By Lemma \ref{lem: sequence}, we can take $\gamma_k = \gamma$ for any fixed $\gamma\in\Gamma$, then $H$ commutes with $\Gamma$, and thus is projected to a conjugacy between $f$ and $\Psi$.

\begin{corollary}\label{cor: density of FuGamma}
	$\widetilde{\mathcal{F}}^u(\Gamma) = \bigcup_{\gamma\in\Gamma}\widetilde{\mathcal{F}}^u(\gamma)$ is dense in $N$.
\end{corollary}

\begin{proof}
	There is a proof in \cite[Lemma 3.12]{Moosavi2019}. Here we give another proof. First we consider the algebraic case. The unstable leaf $\mathcal{L}^u(e\Gamma)$ is dense in $M = N/\Gamma$, because $\Psi$ is topologically transitive (see Lemma \ref{lem: transitivity}) and $e\Gamma$ is a fixed point of $\Psi$. Now $\widetilde{\mathcal{L}}^u(\Gamma) = \pi^{-1}(\mathcal{L}^u(e\Gamma))$ and $\pi(\widetilde{\mathcal{L}}^u(\gamma)) = \mathcal{L}^u(e\Gamma)$, $\forall\gamma\in\Gamma$, for any open subset $U$ of $N$, there exists $\gamma\in\Gamma$ such that $(U\gamma)\bigcap\widetilde{\mathcal{L}}^u(e)\not=\varnothing$, and hence by right-$\Gamma$-invariance of $\widetilde{\mathcal{L}}^u$, we have $\widetilde{\mathcal{L}}^u(\gamma^{-1})\bigcap U\not=\varnothing$, thus $\widetilde{\mathcal{L}}^u(\Gamma)$ is dense in $N$.
    
    Therefore, $\Psi^k(\widetilde{\mathcal{L}}^u(\Gamma)) = \widetilde{\mathcal{L}}^u(\Psi^k\Gamma)$ is also dense in $N$, for $k \geq 1$. As a result, for any $y\in N$, there exists $x_k\in\widetilde{\mathcal{L}}^u(e)$ and $\gamma_k\in\Psi^k\Gamma$ such that $x_k\gamma_k\to y$. By Lemma \ref{lem: sequence} and uniform continuity of $H^{-1}$, we have
	$$d(H^{-1}(x_k\gamma_k), H^{-1}(x_k)\gamma_k)\to 0, \text{ and } d(H^{-1}(x_k\gamma_k), H^{-1}(y))\to 0.$$
	Thus $d(H^{-1}(x_k)\gamma_k, H^{-1}(y)) \to 0$, where $H^{-1}(x_k)\in \widetilde{\mathcal{F}}^u(e)$, hence $\widetilde{\mathcal{F}}^u(e)\Gamma = \widetilde{\mathcal{F}}^u(\Gamma)$ is dense in $N$.
\end{proof}

\begin{lemma}\label{lem: Holder}
	Let $H$ be as in Lemma \ref{lem: conjugate}. Then $H$ and $H^{-1}$ are both H\"older continuous.
\end{lemma}

\begin{proof}
	There are constants $0 < \mu^s_-(\Psi) \leq \mu^s_+(\Psi) < 1 < \mu^u_-(\Psi) \leq \mu^u_+(\Psi)$ such that
	$$\mu^s_-(\Psi)^kd_{\widetilde{\mathcal{L}}^s}(x, y) \leq d_{\widetilde{\mathcal{L}}^s}(\Psi^k(x), \Psi^k(y)) \leq \mu^s_+(\Psi)^kd_{\widetilde{\mathcal{L}}^s}(x, y), \ \forall y\in\widetilde{\mathcal{L}}^s(x),\ \forall k \geq 0;$$
	$$\mu^u_-(\Psi)^kd_{\widetilde{\mathcal{L}}^u}(x, y) \leq d_{\widetilde{\mathcal{L}}^u}(\Psi^k(x), \Psi^k(y)) \leq \mu^u_+(\Psi)^kd_{\widetilde{\mathcal{L}}^u}(x, y), \ \forall y\in\widetilde{\mathcal{L}}^u(x),\ \forall k \geq 0.$$
	There are also constants $0 < \mu^s_-(F) \leq \mu^s_+(F) < 1 < \mu^u_-(F) \leq \mu^u_+(F)$ and $C_1 > 1$ such that for $y\in \widetilde{\mathcal{F}}^s(x)$, we have
	$$C_1^{-1}\mu^s_-(F)^kd_{\widetilde{\mathcal{F}}^s}(x, y) \leq d_{\widetilde{\mathcal{F}}^s}(F^k(x), F^k(y)) \leq C_1\mu^s_+(F)^kd_{\widetilde{\mathcal{F}}^s}(x, y),\ \forall k \geq 0;$$
    and for $y\in \widetilde{\mathcal{F}}^u(x)$, we have
	$$C_1^{-1}\mu^u_-(F)^kd_{\widetilde{\mathcal{F}}^u}(x, y) \leq d_{\widetilde{\mathcal{F}}^u}(F^k(x), F^k(y)) \leq C_1\mu^u_+(F)^kd_{\widetilde{\mathcal{F}}^u}(x, y),\ \forall k \leq 0,$$
	since $\pi\circ F = f$, $\pi$ is locally isometric, $M$ is compact.
	Since $d(H, \text{\rm Id}_N) \leq C_0$, we have
	$$\mu^s_-(F) \leq \mu^s_+(\Psi),\ \mu^s_-(\Psi) \leq \mu^s_+(F),\ \mu^u_-(F) \leq \mu^u_+(\Psi),\ \mu^u_-(\Psi) \leq \mu^u_+(F).$$
	By uniform continuity of $H$, there exists $\delta > 0$ such that $d(x, y) < \delta$ implies $d(H(x), H(y)) < 1$. By right-invariance of $d$ and $d_{\widetilde{\mathcal{L}}^s}$, there exists $C_2 > 1$ such that $d(x, y) < 1$ implies $d(x, y) \leq d_{\widetilde{\mathcal{L}}^s}(x, y) \leq C_2d(x, y)$.
	
	Take $y\in \widetilde{\mathcal{F}}^s(x)$ and $d_{\widetilde{\mathcal{F}}^s}(x, y) < \delta$, and let $N\geq 0$ be the integer such that
	$$d_{\widetilde{\mathcal{F}}^s}(F^{-N}(x), F^{-N}(y)) < \delta, \text{ and } d_{\widetilde{\mathcal{F}}^s}(F^{-N - 1}(x), F^{-N - 1}(y)) \geq \delta.$$
	It follows that
	\begin{align*}
		&d(F^{-N}(x), F^{-N}(y)) < \delta,\\
		&d(\Psi^{-N}(H(x)), \Psi^{-N}(H(y))) = d(H(F^{-N}(x)), H(F^{-N}(y))) < 1,\\
		&d_{\widetilde{\mathcal{L}}^s}(\Psi^{-N}(H(x)), \Psi^{-N}(H(y))) < C_2,\\
		&d(H(x), H(y)) \leq d_{\widetilde{\mathcal{L}}^s}(H(x), H(y)) < C_2\mu^s_+(\Psi)^N,\\
		&d_{\widetilde{\mathcal{F}}^s}(x, y) \geq C^{-1}\mu^s_-(F)^{N + 1}d_{\widetilde{\mathcal{F}}^s}(F^{-N - 1}(x), F^{-N - 1}(y)) \geq C_1^{-1}\mu^s_-(F)^{N + 1}\delta.
	\end{align*}
	As a result, we have $d(H(x), H(y)) \leq Cd_{\widetilde{\mathcal{F}}^s}(x, y)^\alpha$ for $C = C_1C_2^\alpha\delta^{-\alpha}\mu^s_+(\Psi) > 1$ and $\alpha = \frac{\ln \mu^s_+(\Psi)}{\ln \mu^s_-(F)} \in (0, 1]$.
	
	Similarly, when $y\in \widetilde{\mathcal{F}}^u(x)$ and $d_{\widetilde{\mathcal{F}}^u}(x, y) < \delta$, we have $d(H(x), H(y)) \leq C'd_{\widetilde{\mathcal{F}}^u}(x, y)^{\alpha'}$, where $C' > 1$ and $\alpha' = \frac{\ln \mu^u_-(\Psi)}{\ln\mu^u_+(F)} \in (0, 1]$. Without loss of generality, assume that $C \geq C'$ and $\alpha \leq\alpha'$.
	
	Generally, for fixed $x\in N$, take $\delta' > 0$ such that $d(x, y) < \delta'$ implies that $\text{diam}\{x, y, z\} < \delta$, where $z = \beta_{\widetilde{\mathcal{F}}}(x, y)$. When $d(x, y) < \delta'$, we have
	\begin{align*}
		d(H(x), H(y)) &\leq d(H(x), H(z)) + d(H(z), H(y))\\
		&\leq Cd_{\widetilde{\mathcal{F}}^s}(x, z)^\alpha + C'd_{\widetilde{\mathcal{F}}^u}(z, y)^{\alpha'}\\
		&\leq 2^{1 - \alpha}C(d_{\widetilde{\mathcal{F}}^s}(x, z) + d_{\widetilde{\mathcal{F}}^u}(z, y))^\alpha\notag \leq C''d(x, y)^\alpha.
	\end{align*}
	The last inequality holds for some constant $C'' > 1$ because $T_xN = \widetilde{E}^s(x)\oplus\widetilde{E}^u(x)$. We note that $C''$ and $\delta'$ depends on $x$, since $\widetilde{\mathcal{F}}^u$ might not be right-$\Gamma$-invariant.
	Similar argument works for $H^{-1}$.
\end{proof}

\subsection{Stable Lyapunov exponents}

In this subsection, $M = N/\Gamma$ is a nilmanifold with a Riemannian metric induced by a right-invariant Riemannian metric on $N$, $f$ is an Anosov map on $M$ with one-dimensional stable bundle, $F$ is a lift of $f$, $\Psi$ is the linear part of $F$, $H$ is the conjugacy between $F$ and $\Psi$ constructed in Lemma \ref{lem: conjugate}. We will prove that the constant stable Lyapunov exponent at periodic points of $f$ is in fact equal to $\Psi$.

First we state shadowing lemma for Anosov maps on closed manifolds and transitivity for Anosov maps on nilmanifolds.

\begin{lemma}\label{lem: shadowing}
	{\rm (\cite[Theorem 1.2.1]{Aoki1994})} Let $M$ be a closed manifold and $f$ be an Anosov map of $M$, then the followings hold.
	
	\begin{enumerate}
		\item There exists $\varepsilon_0 > 0$ such that if two orbits $(x_i)$ and $(y_i)$ of $f$ satisfy $d(x_i, y_i) \leq \varepsilon_0$ for $i\in\mathbb{Z}$, then $(x_i) = (y_i)$.
		\item For any $\varepsilon > 0$, there exists $\delta > 0$ such that every $\delta$-pseudo orbit $(x_i)$ (i.e., $d(x_{i + 1}, f(x_i)) < \delta$) is $\varepsilon$-traced by some orbit $(y_i)$ (i.e., $d(x_i, y_i) < \varepsilon$). Here $(x_i)$ can be finite or infinite.
	\end{enumerate}
\end{lemma}

As a corollary, when $2\varepsilon < \varepsilon_0$, the $\varepsilon$-tracing orbit for an infinite pseudo orbit is unique. In particular, for sufficiently small $\delta$ such that $2\varepsilon < \varepsilon_0$, a periodic $\delta$-pseudo orbit is $\varepsilon$-traced by a periodic orbit with the same period.

\begin{lemma}\label{lem: transitivity}
	Nilmanifold Anosov maps are topologically transitive.
\end{lemma}

\begin{proof}
	\cite[Lemma 5.4]{Sumi1996} claims that the non-wandering set $\Omega(f) = M$. By spectral decomposition theorem due to Smale, and since $M$ is connected, we have that $M$ is actually a basic set. Consequently, $f$ is topologically transitive.
\end{proof}

Using these two lemmas, by the same construction in \cite[Claim 2.20]{An2023}, we get an adapted Riemannian metric for $f$ with respect to periodic data.

Let $\text{Per}(f)$ be the set of periodic points of $f$. For a periodic point $p\in\text{Per}(f)$ with period $N_p$, the {\it stable Lyapunov exponent} of $p$ is defined by
\[\lambda^s(p, f) := \ln\mu^s(p, f) := \ln\left\|Df^{N_p}|_{E^s(p)}\right\|^{\frac{1}{N_p}} = \frac{1}{N_p}\sum_{i = 0}^{N_p - 1}\ln\left\|Df|_{E^s(f^i(p))}\right\|.\]
For the algebraic case, notice that
\[\left\|D_{\pi(x)}\Psi|_{L^s(\pi(x))}\right\| = \left\|D_x\Psi|_{\widetilde{L}^s(x)}\right\| = \left\|D_x\Psi\circ D_eR_x|_{\mathfrak{n}^s}\right\| = \left\|D_eR_{\Psi(x)}\circ \psi|_{\mathfrak{n}^s}\right\| = \mu^s(\Psi),\]
which is the modulus of eigenvalue of $\psi = D_e\Psi$ along $\mathfrak{n}^s$, independent of $x\in N$. Thus $\lambda^s(p, \Psi) = \lambda^s(\Psi) = \ln\mu^s(\Psi)$, independent of $p\in\text{Per}(\Psi)$.
Let \[\mu_+ := \sup\{\mu^s(p, f): p\in\text{Per}(f)\}\quad {\rm and} \quad \mu_- := \inf\{\mu^s(p, f): p\in \text{Per}(f)\}.\] Clearly $0 < \mu_- \leq \mu_+  < 1$.

\begin{lemma} \label{lem: adapted metric}
    {\rm(\cite[Claim 2.20]{An2023})}For any $\delta > 0$, there exists a Riemannian metric on $M$ such that the induced norm $\left|\cdot\right|$ satisfies $\mu_-(1 + \delta)^{-1} < |Df|_{E^s(x)}| < \mu_+(1 + \delta)$, $\forall x\in M$.
\end{lemma}

We also need the following quasi-norm.

Let $\{\mathfrak{n}_{(i)}\}$ be any fixed subspaces of $\mathfrak{n}$, such that $\mathfrak{n}_i = \mathfrak{n}_{(i)} \bigoplus \mathfrak{n}_{i + 1}$, $1 \leq i \leq s$. Then $\mathfrak{n} = \bigoplus_{i = 1}^s\mathfrak{n}_{(i)}$. Denote the projection from $\mathfrak{n}$ to $\mathfrak{n}_{(i)}$ by $p_i$. Define $q(X) = \max_{1\leq i \leq s}\left\|p_i(X)\right\|^{\frac{1}{i}}$. For $x$, $y\in N$, define $\rho(x, y) = q(\exp^{-1}(yx^{-1}))$. Clearly $\rho$ is continuous, and
\begin{itemize}
	\item $\rho(x, y) = 0$ if and only if $x = y$, $\forall x$, $y\in N$;
	\item $\rho(x, y) = \rho(y, x)$, $\forall x$, $y\in N$;
	\item $\rho(xz, yz) = \rho(x, y)$, $\forall x$, $y$, $z\in N$;
	\item $t^{-\frac{1}{i}}\rho(e, \exp tX) \to \left\|p_i(X)\right\|^{\frac{1}{i}} > 0$ as $t \to +\infty$, when $X\in\mathfrak{n}_i\setminus\mathfrak{n}_{i + 1}$.
\end{itemize}

The following lemma is a direct corollary from \cite[Theorem 2.7, Proposition 4.4]{Breuillard2014}.

\begin{lemma}\label{lem: rho and d}
	{\rm (\cite{Breuillard2014})} There exists $C > 0$ such that $C^{-1}d(x, y) - C \leq \rho(x, y) \leq Cd(x, y) + C$, $\forall x$, $y\in N$.
\end{lemma}

For comparing the dynamics between $F$ and $\Psi$, a strong tool is quasi-isometry. A foliation $\mathcal{F}$ on $N$ is {\it quasi-isometric}, if there exists $C > 0$ such that
\[d(x, y) \leq d_{\mathcal{F}}(x, y) \leq Cd(x, y) + C, \quad  \forall x\in N, \forall y\in \mathcal{F}(x).\]
The one-dimensional stable foliation $\widetilde{\mathcal{L}}^s$ is quasi-isometric. Take $C_s > 0$ such that
\[d(x, y) \leq d_{\widetilde{\mathcal{L}}^s}(x, y) < C_sd(x, y) + C_s.\]
Indeed, this follows from Lemma\ref{lem: rho and d}. Since $\widetilde{\mathcal{L}}^s(e) = \exp\mathfrak{g}^s$, take a unit vector $X$ in $\mathfrak{g}^s$, then $X\in\mathfrak{n}\setminus\mathfrak{n}_2$, hence $t^{-1}\rho(e, \exp tX) \to 1$ as $t \to+\infty$. Notice that $t = d_{\widetilde{\mathcal{L}}^s}(e, \exp tX)$. Therefore, $\widetilde{\mathcal{L}}^s$ is quasi-isometric. 

\begin{lemma}\label{lem: quasi-isometry}
	The one-dimensional foliation $\widetilde{\mathcal{F}}^s$ is quasi-isometric.
\end{lemma}

\begin{proof}
	First we prove two claims.
	
	\begin{claim*}
		For any $C > 0$, there exist positive constants $\varepsilon_k(C) \to 0$ as $k \to +\infty$ such that for any sequence $\gamma_k\in\Psi^k\Gamma$, we have $d(\beta_{\widetilde{\mathcal{F}}}(x\gamma_k, y\gamma_k), \beta_{\widetilde{\mathcal{F}}}(x, y)\gamma_k) \leq \varepsilon_k(C)$, $\forall x, y\in N$, $d(x, y) \leq C$.
	\end{claim*}
	
	\begin{proof}[Proof of Claim]
		Recall the proof of Lemma \ref{lem: sequence}. By right-invariance of $\beta$ and $d$, there are positive constants $\delta_k(C) \to 0$ as $k \to +\infty$, such that for any $x, y\in N$ with $d(x, y) \leq C$, $d(z, y) \leq 2C_0\mu^s(\Psi)^k$ implies that $d(\beta(x, z), \beta(x, y)) \leq \delta_k(C)$, $\forall x$, $y\in N$. Therefore,
		\begin{align*}
			&d(H(\beta_{\widetilde{\mathcal{F}}}(x\gamma_k, y\gamma_k)), H(\beta_{\widetilde{\mathcal{F}}}(x, y)\gamma_k))\\
			= &d(H(\beta_{\widetilde{\mathcal{F}}}(x\gamma_k, y\gamma_k))\gamma_k^{-1}, H(\beta_{\widetilde{\mathcal{F}}}(x, y)\gamma_k)\gamma_k^{-1})\\
			\leq & d(\beta(H(x\gamma_k), H(y\gamma_k))\gamma_k^{-1}, H(\beta_{\widetilde{\mathcal{F}}}(x, y))) + 2C_0\mu^s(\Psi)^k\\
			= &d(\beta(H(x\gamma_k)\gamma_k^{-1}, H(y\gamma_k)\gamma_k^{-1}), \beta(H(x), H(y))) + 2C_0\mu^s(\Psi)^k\\
			= & d(\beta(H(x), H(y\gamma_k)\gamma_k^{-1}), \beta(H(x), H(y))) + 2C_0\mu^s(\Psi)^k\\
			\leq & \delta_k(C + 2C_0) + 2C_0\mu^s(\Psi)^k.
		\end{align*}
		Note that by Lemma \ref{lem: H and Gamma}, $H(x\gamma)\gamma^{-1}\in\widetilde{\mathcal{L}}^s(H(x))$.
		
		Again by uniform continuity of $H^{-1}$, there are positive constants $\varepsilon_k(C) \to 0$ as $k \to +\infty$, such that $d(\beta_{\widetilde{\mathcal{F}}}(x\gamma_k, y\gamma_k), \beta_{\widetilde{\mathcal{F}}}(x, y)\gamma_k) \leq \varepsilon_k(C)$.
	\end{proof}
	
	\begin{claim*}
		For any $C > 0$, there exists $K(C) > 0$ such that $d(x, y) < C$ implies that $d_{\widetilde{\mathcal{F}}^s}(x, \beta_{\widetilde{\mathcal{F}}}(x, y)) < K(C)$.
	\end{claim*} 
	
	\begin{proof}[Proof of Claim]
		Fix $k$ large such that $\varepsilon_k(C) \leq 1$. Since $\Gamma$ is cocompact, there exists a compact subset $S\subseteq N$ satisfying $S\Gamma = N$, and hence $\Psi^k(S)\Psi^k\Gamma = N$. Let $\Omega = \overline{B(\Psi^k(S), C)}$, which is compact, and $K_C := \max\{d_{\widetilde{\mathcal{F}}^s}(x, \beta_{\widetilde{\mathcal{F}}}(x, y)): x, y\in \Omega\}$. For any $x, y\in N$, $d(x, y) < C$, take $\gamma_k\in\Psi^k\Gamma$ such that $x\gamma_k\in \Psi^k(S)\subseteq\Omega$, then $y\gamma_k \in \Omega$ since $d(x\gamma_k, y\gamma_k) = d(x, y) < C$. It follows that $d_{\widetilde{\mathcal{F}}^s}(x\gamma_k, \beta_{\widetilde{\mathcal{F}}}(x\gamma_k, y\gamma_k)) \leq K_C$ and hence
		\begin{align*}
			d_{\widetilde{\mathcal{F}}^s}(x, \beta_{\widetilde{\mathcal{F}}}(x, y)) &= d_{\widetilde{\mathcal{F}}^s}(x\gamma_k, \beta_{\widetilde{\mathcal{F}}}(x, y)\gamma_k)\\
			&\leq d_{\widetilde{\mathcal{F}}^s}(x\gamma_k, \beta_{\widetilde{\mathcal{F}}}(x\gamma_k, y\gamma_k)) + d(\beta_{\widetilde{\mathcal{F}}}(x\gamma_k, y\gamma_k), \beta_{\widetilde{\mathcal{F}}}(x, y)\gamma_k)\\
            &\leq K_C + 1.
		\end{align*}
		Note that $\widetilde{\mathcal{F}}^s$ is right-$\Gamma$-invariant. Finally we take $K(C) = K_C + 1$.
	\end{proof}
	Return to the lemma. Recall that $d_{\widetilde{\mathcal{L}}^s}(x, y) \leq C_sd(x, y) + C_s$, $\forall x\in N$, $\forall y\in \widetilde{\mathcal{L}}^s(x)$, for some $C_s > 0$.
	
	Take $v\in \widetilde{\mathcal{F}}^s(e)$ such that $H(v)\in \widetilde{\mathcal{L}}^s(e)$ and $d(e, H(v)) = 3C_0$. Such $v$ exists, because $H: \widetilde{\mathcal{F}}^s(e) \to \widetilde{\mathcal{L}}^s(e)$ is a homeomorphism, $\exp: \mathfrak{n}^s \to \widetilde{\mathcal{L}}^s(e)$ is a diffeomorphism, $\widetilde{\mathcal{L}}^s$ is quasi-isometric, and further, $d(H, \text{\rm Id}_N) \leq C_0$. It follows that $d(H(x), H(v)H(x)) = 3C_0$, $\forall x\in N$, and thus $C_0 \leq d(x, v_x) \leq 5C_0$, where $v_x = H^{-1}(H(v)H(x))\in\widetilde{\mathcal{F}}^s(x)$ and hence $v_x = \beta_{\widetilde{\mathcal{F}}}(x, v_x)$.
	
	By the claim above, $d(x, v_x) \leq 5C_0$ implies $d_{\widetilde{\mathcal{F}}^s}(x, \beta_{\widetilde{\mathcal{F}}}(x, v_x)) < K_0$ for some constant $K_0 > 0$, that is, $d_{\widetilde{\mathcal{F}}^s}(x, v_x) < K_0$.
	
	Now consider $x_k = H^{-1}(H(v)^kH(x))$, $k\in\mathbb{Z}$. Note that $x_0 = x$, $x_1 = v_x$. The same argument claims that $d(H(v)^kH(x), H(v)^{k + 1}H(x)) = 3C_0$, $C_0 \leq d(x_k, x_{k + 1}) \leq 5C_0$, $d_{\widetilde{\mathcal{F}}^s}(x_k, x_{k + 1}) < K_0$.
	
	Notice that $H(v)^kH(x) \in \widetilde{\mathcal{L}}^s(H(x))$. It follows that $\{H(v)^kH(x): k\in\mathbb{Z}\}$ divides the curve into intervals with endpoints $H(v)^kH(x)$ and $H(v)^{k + 1}H(x)$. On the other hand, $H^{-1}(\widetilde{\mathcal{L}}^s(H(x))) = \widetilde{\mathcal{F}}^s(x)$, so the curve $\widetilde{\mathcal{F}}^s(x)$ is also divided into intervals with endpoints $x_k$ and $x_{k + 1}$.
	
	For any $y\in \widetilde{\mathcal{F}}^s(x)$, assume that $y$ lies in the interval with endpoints $x_k$ and $x_{k + 1}$. It follows that $d_{\widetilde{\mathcal{F}}^s}(x, y) \leq (|k| + 1)K_0$, and
	\begin{align*}
		d(x, y) &\geq d(H(x), H(y)) - 2C_0 \geq d(H(x), H(v)^kH(x)) - 3C_0 - 2C_0\\
		&= d(e, H(v)^k) - 5C_0 \geq C_s^{-1}(d_{\widetilde{\mathcal{L}}^s}(e, H(v)^k) - C_s) - 5C_0\\
		&= C_s^{-1}(|k|d_{\widetilde{\mathcal{L}}^s}(e, H(v)) - C_s) - 5C_0 \geq C_s^{-1}(|k|d(e, H(v)) - C_s) - 5C_0\\
		&\geq (3C_s^{-1}|k| - 6)C_0.
	\end{align*}
	Take $C = (6C_s + 1)K_0 \geq \frac{C_sK_0}{C_0}$, then $d_{\widetilde{\mathcal{F}}^s}(x, y) \leq Cd(x, y) + C$.
\end{proof}

\begin{theorem}\label{thm: p, q = A}
	If $\lambda^s(p, f) = \lambda^s(q, f)$, for all $ p$, $q\in\text{Per}(f)$, then $\lambda^s(p, f) = \lambda^s(\Psi)$, for all $ p\in\text{Per}(f)$.
\end{theorem}

\begin{proof}
	Under the assumption, we have $\mu_- = \mu_+$ and denote it by $\mu$. By Lemma \ref{lem: adapted metric}, there is an adapted metric on $M$ such that $\mu(1 + \delta)^{-1} < |Df|_{E^s(x)}| < \mu(1 + \delta)$, $\forall x\in M$. This metric induces a metric on $N$ such that $\mu(1 + \delta)^{-1} < |DF|_{\widetilde{E}^s(x)}| < \mu(1 + \delta)$, $\forall x\in N$.
	
	Denote the distance on $N$ induced by the adapted metric by $D$, then $D$ is equivalent to the right-invariant distance $d$. Assume that $C_1^{-1}d\leq D\leq C_1d$ for some constant $C_1 > 0$.
	
	Take $x\in N$, $y\in\widetilde{\mathcal{F}}^s(x)$, $y\not= x$, then
	$$\mu^{-k}(1 + \delta)^{-k}D_{\widetilde{\mathcal{F}}^s}(x, y) \leq D_{\widetilde{\mathcal{F}}^s}(F^{-k}(x), F^{-k}(y)) \leq \mu^{-k}(1 + \delta)^kD_{\widetilde{\mathcal{F}}^s}(x, y),$$
	hence
	$$\mu^{-k}(1 + \delta)^{-k}C_1^{-2}d_{\widetilde{\mathcal{F}}^s}(x, y) \leq d_{\widetilde{\mathcal{F}}^s}(F^{-k}(x), F^{-k}(y)) \leq \mu^{-k}(1 + \delta)^kC_1^2d_{\widetilde{\mathcal{F}}^s}(x, y).$$
	
	Since $\widetilde{\mathcal{F}}^s$ is quasi-isometric by Lemma \ref{lem: quasi-isometry}, $d(x, y) \leq d_{\widetilde{\mathcal{F}}^s}(x, y) \leq C_2d(x, y) + C_2$ for some constant $C_2 > 0$. Therefore,
	$$C_2^{-1}\mu^{-k}(1 + \delta)^{-k}C_1^{-2}d_{\widetilde{\mathcal{F}}^s}(x, y) - 1\leq d(F^{-k}(x), F^{-k}(y)) \leq \mu^{-k}(1 + \delta)^kC_1^2d_{\widetilde{\mathcal{F}}^s}(x, y).$$
	
	Recall that $d(H, \text{\rm Id}_N) \leq C_0$. Hence
	\begin{align*}
		C_2^{-1}\mu^{-k}(1 + \delta)^{-k}C_1^{-2}d_{\widetilde{\mathcal{F}}^s}(x, y) - 2C_0 - 1
		&\leq d(H(F^{-k}(x)), H(F^{-k}(y)))\\
		&\leq \mu^{-k}(1 + \delta)^kC_1^2d_{\widetilde{\mathcal{F}}^s}(x, y) + 2C_0.
	\end{align*}
	
	On the other hand,
	\begin{align*}
        d_{\widetilde{\mathcal{L}}^s}(H(F^{-k}(x)), H(F^{-k}(y)))
        &= d_{\widetilde{\mathcal{L}}^s}(\Psi^{-k}(H(x)), \Psi^{-k}(H(y)))\\
        &= \mu^s(\Psi)^{-k}d_{\widetilde{\mathcal{L}}^s}(H(x), H(y)).
    \end{align*}
	
	Since $d_{\widetilde{\mathcal{L}}^s}$ is also quasi-isometric, we have
	\begin{align*}
	    C_s^{-1}\mu^s(\Psi)^{-k}d_{\widetilde{\mathcal{L}}^s}(H(x), H(y)) - 1
     &\leq d(H(F^{-k}(x)), H(F^{-k}(y)))\\
     &\leq \mu^s(\Psi)^{-k}d_{\widetilde{\mathcal{L}}^s}(H(x), H(y)).
	\end{align*}
	
	Fix $x$ and $y$, let $k\to +\infty$. It follows that $\mu(1 + \delta)^{-1} \leq \mu^s(\Psi) \leq \mu(1 + \delta)$. Since $\delta$ can be arbitrarily small, one has $\mu = \mu^s(\Psi)$, and the proof is completed.
\end{proof}

\subsection{Special Anosov maps on nilmanifolds}\label{special Anosov}

In this subsection, $M = N/\Gamma$ is a nilmanifold with a Riemannian metric induced by a right-invariant Riemannian metric on $N$, $f$ is an Anosov map on $M$, $F$ is a lift of $f$, $\Psi$ is the linear part of $F$, $H$ is the conjugacy between $F$ and $\Psi$ constructed in Lemma \ref{lem: conjugate}.

\begin{theorem}\label{thm: special and conjugacy}
	The following statements hold.
	\begin{enumerate}
		\item If $f$ is topologically conjugate to $\Psi$, then $f$ is special.
		\item If $f$ is special and $\Psi$ is $u$-ideal, then $f$ is topologically conjugate to $\Psi$.
	\end{enumerate}
\end{theorem}

Note that, $f$ is special if and only if $\widetilde{\mathcal{F}}^u$ is right-$\Gamma$-invariant. Recall that $u$-ideal means that $\mathfrak{n}^u$ is an ideal, or equivalently, $[\mathfrak{n}^s, \mathfrak{n}^u] \subseteq\mathfrak{n}^u$. Anosov maps on tori, and Anosov maps on nilmanifolds with one-dimensional stable bundle, these two cases satisfy this condition, see Remark \ref{rmk: codimension one to u-ideal}.

\begin{proposition}\label{prop: unique H}
	If $f$ is topologically conjugate to $\Psi$, then $H$ commutes with $\Gamma$ and is projected to a conjugacy $h$ between $f$ and $\Psi$.
\end{proposition}

\begin{proof}
	Assume that $f$ is topologically conjugate to its linear part $\Psi$ via some homeomorphism $\widehat{h}$. Take a lift $\widehat{H}$ of $\widehat{h}$, then there exists $\gamma_0\in\Gamma$ such that $\widehat{H}\circ F = R_{\gamma_0}\circ\Psi\circ \widehat{H}$. By Lemma \ref{lem: affine conjugate auto}, there exists $n\in N$ such that $R_n\circ\widehat{H}\circ F = \Psi\circ R_n\circ\widehat{H}$. Specifically, $\Psi(n)n^{-1} = \gamma_0$.
	
	Let $A$ be the linear part of $\widehat{H}$, since $\widehat{H}$ is a lift of $\widehat{h}$. Then $R_n\circ\widehat{H}(x\gamma) = \widehat{H}(x)A(\gamma)n = (R_n\circ\widehat{H}(x))\cdot (Ad_n\circ A(\gamma))$ . It follows that $H_0 := (Ad_n\circ A)^{-1}\circ R_n\circ\widehat{H}$ satisfies $H_0(x\gamma) = H_0(x)\gamma$, $\forall x\in N$, $\gamma\in\Gamma$, which leads to $d(H_0, \text{\rm Id}_N) < +\infty$. Moreover, by calculating $R_n\circ\widehat{H}\circ F(x\gamma) = \Psi\circ R_n\circ\widehat{H}(x\gamma)$, one has that $Ad_n\circ A\circ\Psi(\gamma) = \Psi\circ Ad_n\circ A(\gamma)$, $\forall\gamma\in\Gamma$. Since a monomorphism of $\Gamma$ is uniquely extended to an automorphism of $N$, it follows that $Ad_n\circ A$ commutes with $\Psi$, and hence $H_0\circ F = \Psi\circ H_0$. By uniqueness of Lemma \ref{lem: conjugate}, $H_0 = H$, and thus $H(x\gamma) = H(x)\gamma$.
\end{proof}

\begin{proof}[Proof of Theorem \ref{thm: special and conjugacy}]
	The first item of Theorem \ref{thm: special and conjugacy} is a direct corollary of Proposition \ref{prop: unique H}. Indeed,  when $f$ is conjugate to $\Psi$,
    \[\widetilde{\mathcal{F}}^u(x\gamma) = H^{-1}(\widetilde{\mathcal{L}}^u(x\gamma)) = H^{-1}(\widetilde{\mathcal{L}}^u(x)\gamma) = H^{-1}(\widetilde{\mathcal{L}}^u(x))\gamma = \widetilde{\mathcal{F}}^u(x)\gamma, \]
	hence $f$ is special.
	
	The proof of the second item of Theorem \ref{thm: special and conjugacy} is more complicated and we divided it into several lemmas.
	
	\begin{lemma}\label{lem: LuLu}
		If $\ [\mathfrak{n}^s, \mathfrak{n}^u] \subseteq\mathfrak{n}^u$, then $d(z, \widetilde{\mathcal{L}}^u(y)) = d(x, \widetilde{\mathcal{L}}^u(y))$, $\forall x$, $y\in N$, $z\in\widetilde{\mathcal{L}}^u(x)$.
	\end{lemma}
	
	\begin{proof}[Proof of Lemma \ref{lem: LuLu}]
		Since $d$ and $\widetilde{\mathcal{L}}^u$ are right-invariant, we may assume that $y = e$. Since $\widetilde{\mathcal{L}}^u(x) = \widetilde{\mathcal{L}}^u (\beta(e, x))$, we may replace $x$ by $\beta(e, x)$ and assume at first that $x\in \widetilde{\mathcal{L}}^s(e)$.
		
		Now $\widetilde{\mathcal{L}}^u(x) = \widetilde{\mathcal{L}}^u(e)x$, any $z\in\widetilde{\mathcal{L}}^u(x)$ has the form of $wx$ where $w\in\widetilde{\mathcal{L}}^u(e)$. Therefore,
		\begin{align*}
            d(z, \widetilde{\mathcal{L}}^u(e)) &= d(wx, \widetilde{\mathcal{L}}^u(e)) = d(x[x^{-1}, w]w, \widetilde{\mathcal{L}}^u(e))\\
            &= d(x, \widetilde{\mathcal{L}}^u(w^{-1}[x^{-1}, w]^{-1})) = d(x, \widetilde{\mathcal{L}}^u(e)).
        \end{align*}
		The last equality holds because
        $$w^{-1}[x^{-1}, w]^{-1} = w^{-1}[w, x^{-1}]\in\widetilde{\mathcal{L}}^u(e)[\widetilde{\mathcal{L}}^u(e), \widetilde{\mathcal{L}}^s(e)] \subseteq \widetilde{\mathcal{L}}^u(e).$$
        The proof of Lemma \ref{lem: LuLu} is completed.
	\end{proof}
	
	\begin{lemma}\label{lem: growth}
		Let $K$ be a compact subset of $\ \mathfrak{n}^s$ away from zero, $d(t) :=\inf\{d(\exp tX, \widetilde{\mathcal{L}}^u(e)): X\in K\}$. If $[\mathfrak{n}^s, \mathfrak{n}^u] \subseteq\mathfrak{n}^u$, then there exists $0 < a < 1$ such that $d(t) \geq at^{\frac{1}{s}} - 2$, $\forall t\geq 0$.
	\end{lemma}
	
	\begin{proof}[Proof of Lemma \ref{lem: growth}]
		By Lemma \ref{lem: rho and d}, there exists $C > 0$ such that
		\begin{align*}
			d(\exp tX, \exp tY)
			& \geq C^{-1}\rho(\exp tX, \exp tY) - 1\\
			&= C^{-1}\rho((\exp tX)\exp(-tY), e) - 1\\
			&= C^{-1}\rho(\exp Z(t, X, Y), e) - 1\\
			&= C^{-1}q(Z(t, X, Y)) - 1,
		\end{align*}
		where $Z(t, X, Y) = tZ_1 + t^2Z_2 + \cdots + t^sZ_s$, $Z_1 = X - Y \in \mathfrak{n}_1$, $Z_2 = -\frac{1}{2}[X, Y]\in\mathfrak{n}_2$, $\cdots$, $Z_s\in\mathfrak{n}_s$, are decided by the Baker-Campbell-Hausdorff formula $X\odot Y := X + Y + \frac{1}{2}[X, Y] + \cdots$, satisfying $\exp (X\odot Y) = \exp X\exp Y$. Now take $X\in K$ and $Y\in\mathfrak{n}^u$, we have $Z_i\in\mathfrak{n}^u$, $2\leq i \leq s$, and hence
		\begin{align*}
			d(t)
			&\geq C^{-1}\inf\{q(Z(t, X, Y)): X\in K, Y\in\mathfrak{n}^u\} - 1\\
			&\geq C^{-1}\min\{t, t^{\frac{1}{s}}\}\inf\{q(Z_1 + tZ_2 + \cdots + t^{s - 1}Z_s): X\in K, Y\in\mathfrak{n}^u\} - 1\\
			&\geq C^{-1}\min\{t, t^{\frac{1}{s}}\}\inf\{q(X - W): X\in K, W\in\mathfrak{n}^u\} - 1\\
			&\geq a\min\{t, t^{\frac{1}{s}}\} - 1
		\end{align*}
		for some $0 < a < 1$. The last step holds because of the following argument. Assume for contradiction that there exist $X_k \in K$ and $W_k\in\mathfrak{n}^u$ such that $q(X_k - W_k) \to 0$. Since $K$ is compact and hence $\{W_k\}$ is bounded, by taking a subsequence, assume that $X_k \to X_0\in K$, $W_k \to W_0\in\mathfrak{n}^u$. Then $X_0 = W_0 \in K\bigcap\mathfrak{n}^u$, which is a contradiction.
	\end{proof}
	
	\begin{lemma}\label{lem: geometric}
		If $f$ is special and $\Psi$ is $u$-ideal, then for any $\varepsilon > 0$, there exists $\delta > 0$ such that the followings hold.
		\begin{enumerate}
			\item $x\in\widetilde{\mathcal{L}}^s(y)$ and $d(x, y) > \varepsilon$ implies that $d(\widetilde{\mathcal{L}}^u(x), \widetilde{\mathcal{L}}^u(y)) > \delta$.
			\item $d(x, \widetilde{\mathcal{L}}^u(e)) > \varepsilon$ implies that $d(x^k, \widetilde{\mathcal{L}}^u(e)) > k^{\frac{1}{s}}\delta$, $\forall k \geq 1$.
			\item $d(x, y) < \delta$ implies that $\widetilde{\mathcal{F}}^u(x) \subseteq B_{\varepsilon}(\widetilde{\mathcal{F}}^u(y))$, here $B_\varepsilon(S) := \bigcup_{x\in S}B_\varepsilon(x)$.
		\end{enumerate}
	\end{lemma}
	
	\begin{proof}[Proof of Lemma \ref{lem: geometric}]
		For 1, by right-invariance of $d$ and $\widetilde{\mathcal{L}}^\sigma$, $\sigma = s$, $u$, we may assume that $y = e$. Assume for contradiction that there exists a sequence $x_k\in\widetilde{\mathcal{L}}^s(e)$ satisfying $d(x_k, e) > \varepsilon$ and $d(\widetilde{\mathcal{L}}^u(x_k), \widetilde{\mathcal{L}}^u(e)) \to 0$. By Lemma \ref{lem: LuLu}, $d(x_k, \widetilde{\mathcal{L}}^u(e)) \to 0$. Assume that $x_k = \exp t_kX_k$, $X_k\in\mathfrak{n}^s$, $\left\|X_k\right\| = 1$, $t_k > 0$. Then $t_k\to +\infty$ because any convergent subsequence of $\{x_k\}$ converges to $\widetilde{\mathcal{L}}^s(e)\bigcap\widetilde{\mathcal{L}}^u(e) = \{e\}$, which causes contradiction. Let $K$ be the unit sphere of $\mathfrak{n}^s$. By Lemma \ref{lem: growth}, $d(\exp t_kX_k, \widetilde{\mathcal{L}}^u(e)) \geq at_k^{\frac{1}{s}} - 2 \to +\infty$ as $k \to +\infty$, leading to a contradiction again.
		
		For 2, first notice that we have a decomposition $x = x^sx^u$ by Proposition \ref{prop: su decomposition}, and if $x = \exp X$, $x^s = \exp Y$, $x^u = \exp Z$, then we have $X = Y\odot Z := Y + Z + \frac{1}{2}[Y, Z] + \cdots$ given by Baker-Campbell-Hausdorff formula. Moreover,
		\begin{align*}
			d(x^k, \widetilde{\mathcal{L}}^u(e))
			&= d(\exp kX, \widetilde{\mathcal{L}}^u(e))
			= d(\exp k(Y\odot Z), \widetilde{\mathcal{L}}^u(e))\\
			&= d(\exp (kY + W_k), \widetilde{\mathcal{L}}^u(e))
			= d(\exp kY \exp ((-kY)\odot (kY + W_k)), \widetilde{\mathcal{L}}^u(e))\\
			&= d(\exp kY \exp (W_k + W'_k), \widetilde{\mathcal{L}}^u(e))= d(\exp kY, \widetilde{\mathcal{L}}^u(e))= d((x^s)^k, \widetilde{\mathcal{L}}^u(e)),
		\end{align*}
		where $W_k \in \mathfrak{n}^u + [\mathfrak{n}^s, \mathfrak{n}^u] \subseteq\mathfrak{n}^u$ and $W'_k\in[\mathfrak{n}^s, \mathfrak{n}^u] \subseteq\mathfrak{n}^u$. Therefore, we only need to consider $x\in\widetilde{\mathcal{L}}^s(e)$ and $d(x, \widetilde{\mathcal{L}}^u(e)) > \varepsilon$. Assume for contradiction that there exist $x_m = \exp t_mX_m \in\widetilde{\mathcal{L}}^s(e)$, $d(x_m, \widetilde{\mathcal{L}}^u(e)) > \varepsilon$, $\left\|X_m\right\| = 1$, $t_m > 0$, and $k_m \geq 1$, such that $k_m^{-\frac{1}{s}}d(\exp k_mt_mX_m, \widetilde{\mathcal{L}}^u(e)) \to 0$ as $m \to +\infty$. By lemma \ref{lem: growth},
		$$k_m^{-\frac{1}{s}}d(\exp k_mt_mX_m, \widetilde{\mathcal{L}}^u(e)) \geq at_m^{\frac{1}{s}} - 2k_m^{-\frac{1}{s}}.$$
		If $k_m$ is unbounded, then taking a subsequence leads to a contradiction since $t_m$ has a positive lower bound. If $t_m$ is unbounded, this also leads to a contradiction. Assume that $k_m$ and $t_m$ are both bounded. By taking a subsequence, assume that $k_m = k_0$, $t_m \to t_0 > 0$ and $X_m \to X_0$. Then $\exp k_0t_0X_0 \in \widetilde{\mathcal{L}}^u(e)$, which contradicts with $X_0\in\mathfrak{n}^s$, $\left\|X_0\right\| = 1$.
		
		For 3, notice that $\widetilde{\mathcal{F}}^u(x) \subseteq B_{\varepsilon}(\widetilde{\mathcal{F}}^u(y))$ is equivalent to $\sup\{d(z, \widetilde{\mathcal{F}}^u(y)): z\in \widetilde{\mathcal{F}}^u(x)\} < \varepsilon$. Since $H(\widetilde{\mathcal{F}}^u(x)) = \widetilde{\mathcal{L}}^u(H(x))$, by uniform continuity of $H^{-1}$, there exists $\varepsilon' > 0$ such that $\sup\{d(z, \widetilde{\mathcal{L}}^u(H(y))): z\in \widetilde{\mathcal{L}}^u(H(x))\} < \varepsilon'$ is sufficient for the conclusion. By Lemma \ref{lem: LuLu}, $d(H(x), \widetilde{\mathcal{L}}^u(H(y))) < \varepsilon'$ will suffice. By uniform continuity of $H$, there exists $\delta > 0$ such that $d(x, y) < \delta$ implies that $d(H(x), H(y)) < \varepsilon'$, and hence $d(H(x), \widetilde{\mathcal{L}}^u(H(y))) < \varepsilon'$.
	\end{proof}

	\begin{lemma}\label{lem: Fu bigcap Gamma}
		\rm If $f$ is special and $\Psi$ is $u$-ideal, then the followings hold.
		\begin{enumerate}
			\item $\widetilde{\mathcal{L}}^u(e)\bigcap\Gamma$ and $\widetilde{\mathcal{F}}^u(e)\bigcap\Gamma$ are both subgroups;
			\item $\widetilde{\mathcal{F}}^u(e)\bigcap\Gamma\subseteq \widetilde{\mathcal{L}}^u(e)\bigcap\Gamma$.
			\item For $x\in\widetilde{\mathcal{F}}^u(e)$ and $\gamma\in\widetilde{\mathcal{F}}^u(e)\bigcap\Gamma$, we have $H(x\gamma) = H(x)\gamma$.
		\end{enumerate}
	\end{lemma}
	
	\begin{proof}[Proof of Lemma \ref{lem: Fu bigcap Gamma}]
		For 1, $\widetilde{\mathcal{L}}^u(e) = \exp\mathfrak{n}^u$ and $\mathfrak{n}^u$ is a Lie subalgebra, so $\widetilde{\mathcal{L}}^u(e)$ is a closed Lie subgroup and thus $\widetilde{\mathcal{L}}^u(e)\bigcap\Gamma$ is a subgroup. For $\gamma_1, \gamma_2\in\widetilde{\mathcal{F}}^u(e)\bigcap\Gamma$, we have $$\gamma_1\gamma_2^{-1} \in \widetilde{\mathcal{F}}^u(\gamma_1\gamma_2^{-1}) = \widetilde{\mathcal{F}}^u(\gamma_1)\gamma_2^{-1} = \widetilde{\mathcal{F}}^u(e)\gamma_2^{-1} = \widetilde{\mathcal{F}}^u(\gamma_2)\gamma_2^{-1} = \widetilde{\mathcal{F}}^u(e),$$
		hence $\widetilde{\mathcal{F}}^u(e)\bigcap\Gamma$ is also a subgroup.
		
		For 2, take $\gamma\in \widetilde{\mathcal{F}}^u(e)\bigcap\Gamma$ and we need to show that $\gamma\in\widetilde{\mathcal{L}}^u(e)$. Note that $\gamma^k \in \widetilde{\mathcal{F}}^u(e)\bigcap\Gamma$. Since $d(H, \text{Id}_N) \leq C_0$ and $H(\widetilde{\mathcal{F}}^u(e)) = \widetilde{\mathcal{L}}^u(e)$, we have $d(H(\gamma^k), \gamma^k) \leq C_0$, and consequently, $d(\gamma^k, \widetilde{\mathcal{L}}^u(e)) \leq C_0$. By Lemma \ref{lem: geometric}, $\gamma\in \widetilde{\mathcal{L}}^u(e)$.
		
		For 3, by Lemma \ref{lem: H and Gamma}, $H(x\gamma)\gamma^{-1} \in \widetilde{\mathcal{L}}^s(H(x))$. On the other hand, $x\gamma\in\widetilde{\mathcal{F}}^u(e)\gamma = \widetilde{\mathcal{F}}^u(\gamma) = \widetilde{\mathcal{F}}^u(e)$, so $H(x\gamma)\gamma^{-1}\in H(\widetilde{\mathcal{F}}^u(e))\gamma^{-1} = \widetilde{\mathcal{L}}^u(e)\gamma^{-1}$. Now $\gamma\in\widetilde{\mathcal{F}}^u(e)\bigcap\Gamma \subseteq\widetilde{\mathcal{L}}^u(e)\bigcap\Gamma$ by Lemma \ref{lem: Fu bigcap Gamma}, we have $\widetilde{\mathcal{L}}^u(e) = \widetilde{\mathcal{L}}^u(\gamma)$ and thus $H(x\gamma)\gamma^{-1} \in \widetilde{\mathcal{L}}^u(\gamma)\gamma^{-1} = \widetilde{\mathcal{L}}^u(e) = \widetilde{\mathcal{L}}^u(H(x))$, since $H(x) \in \widetilde{\mathcal{L}}^u(e)$. Therefore, $H(x\gamma)\gamma^{-1} \in \widetilde{\mathcal{L}}^s(H(x))\bigcap\widetilde{\mathcal{L}}^u(H(x)) = \{H(x)\}$.
	\end{proof}
	
	\begin{lemma}\label{lem: H'}
		If $f$ is special and $\Psi$ is $u$-ideal, then $H'(x) := H(x\gamma^{-1})\gamma: \widetilde{\mathcal{F}}^u(\gamma) \to \widetilde{\mathcal{L}}^u(\gamma)$ is well-defined for any fixed $\gamma\in\Gamma$.
	\end{lemma}
	
	\begin{proof}[Proof of Lemma \ref{lem: H'}]
		Notice that $\widetilde{\mathcal{F}}^u(\gamma)\gamma^{-1} = \widetilde{\mathcal{F}}^u(e)$, $H(\widetilde{\mathcal{F}}^u(e)) = \widetilde{\mathcal{L}}^u(H(e)) = \widetilde{\mathcal{L}}^u(e)$ and $\widetilde{\mathcal{L}}^u(e)\gamma = \widetilde{\mathcal{L}}^u(\gamma)$, therefore $H'(\widetilde{\mathcal{F}}^u(\gamma)) = \widetilde{\mathcal{L}}^u(\gamma)$. To show that $H'$ is well-defined, it suffices to show that if $\widetilde{\mathcal{F}}^u(\gamma_1) = \widetilde{\mathcal{F}}^u(\gamma_2)$, then $H(x\gamma_1^{-1})\gamma_1 = H(x\gamma_2^{-1})\gamma_2, \forall x\in\widetilde{\mathcal{F}}^u(\gamma_1)$.
		
		Since $x\gamma_2^{-1}\in\widetilde{\mathcal{F}}^u(e)$ and $\gamma_2\gamma_1^{-1}\in\widetilde{\mathcal{F}}^u(e)\bigcap\Gamma$, by Lemma \ref{lem: Fu bigcap Gamma}, we have
		\[H(x\gamma_1^{-1})\gamma_1 = H(x\gamma_2^{-1}\gamma_2\gamma_1^{-1})\gamma_1 = H(x\gamma_2^{-1})\gamma_2\gamma_1^{-1}\gamma_1 = H(x\gamma_2^{-1})\gamma_2.\]
		Hence $H'$ is well-defined.
	\end{proof}
	
	Now we can define $H': \widetilde{\mathcal{F}}^u(\Gamma) \to \widetilde{\mathcal{L}}^u(\Gamma)$. By Corollary \ref{cor: density of FuGamma}, $\widetilde{\mathcal{F}}^u(\Gamma)$ is dense in $N$.
	
	\begin{lemma}\label{lem: uniform continuity of H'}
		$H'$ is uniformly continuous on $\widetilde{\mathcal{F}}^u(\Gamma)$.
	\end{lemma}
	
	\begin{proof}[Proof of Lemma \ref{lem: uniform continuity of H'}]
		Assume for contradiction that there exists $\varepsilon_0 > 0$ such that for any $\delta > 0$ there exist two points $x, y \in \bigcup_{\gamma\in\Gamma}\widetilde{\mathcal{F}}^u(\gamma)$ satisfying $d(x, y) < \delta$ and $d(H'(x), H'(y)) > 2\varepsilon_0$. Assume that $x\in \widetilde{\mathcal{F}}^u(\gamma_x)$, $y\in\widetilde{\mathcal{F}}^u(\gamma_y)$.
		
		By uniform continuity of $H$, there exists $\delta_0 > 0$ such that $d(x, y) < \delta_0$ implies $d(H(x), H(y)) < \varepsilon_0$. By right-invariance of $d$ and $\widetilde{\mathcal{L}}^s$, there exists $\delta_1 > 0$ such that $d(x, y) < \delta_1$ implies $\text{diam}\{x, y, \beta(x, y)\} < \delta_0$. Hence when $\delta < \delta_1$,
		\begin{align*}
			d(H'(x), H'(\beta(x, y))) &\geq d(H'(x), H'(y)) - d(H'(y), H'(\beta(x, y)))\\
			& > 2\varepsilon_0 - d(H(y\gamma_y)\gamma_y^{-1}, H(\beta(x, y)\gamma_y)\gamma_y^{-1})\\
			& = 2\varepsilon_0 - d(H(y\gamma_y), H(\beta(x, y)\gamma_y))\notag > 2\varepsilon_0 - \varepsilon_0 = \varepsilon_0.
		\end{align*}
		Note that $H'(x)\in\widetilde{\mathcal{L}}^u(\gamma_x)$ and $H'(\beta(x, y))\in\widetilde{\mathcal{L}}^u(\gamma_y)$. Hence by Lemma \ref{lem: geometric}, there exists $\varepsilon'_0 > 0$ such that $d(\widetilde{\mathcal{L}}^u(\gamma_x), \widetilde{\mathcal{L}}^u(\gamma_y)) > \varepsilon'_0$.
		
		Take $\gamma = \gamma_x\gamma_y^{-1}$, then we have $d(\widetilde{\mathcal{L}}^u(\gamma), \widetilde{\mathcal{L}}^u(e)) > \varepsilon'_0$. In particular, $d(\gamma, \widetilde{\mathcal{L}}^u(e)) > \varepsilon'_0$. By Lemma \ref{lem: geometric}, there exists $\varepsilon_1 > 0$ such that $d(\gamma^k, \widetilde{\mathcal{L}}^u(e)) > k^{\frac{1}{s}}\varepsilon_1$, for $k \geq 1$.
		
		On the other hand, $x\in \widetilde{\mathcal{F}}^u(\gamma_x)$, $y\in\widetilde{\mathcal{F}}^u(\gamma_y)$, $d(x, y) < \delta$.
		
		Fix $\varepsilon_2 > 0$ sufficiently small, such that there exists a positive integer $k_0$ such that $\frac{C_0}{\varepsilon_2} > k_0 > (\frac{2C_0}{\varepsilon_1})^s$. By Lemma \ref{lem: geometric}, there exists $\delta_2 > 0$ such that when $\delta < \min\{\delta_1, \delta_2\}$, we have $\widetilde{\mathcal{F}}^u(\gamma_x)\subseteq B_{\varepsilon_2}(\widetilde{\mathcal{F}}^u(\gamma_y))$, or equivalently, $\widetilde{\mathcal{F}}^u(\gamma)\subseteq B_{\varepsilon_2}(\widetilde{\mathcal{F}}^u(e))$. By right-$\Gamma$-invariance of $\widetilde{\mathcal{F}}^u$, we have $\widetilde{\mathcal{F}}^u(\gamma^k) \subseteq B_{\varepsilon_2}(\widetilde{\mathcal{F}}^u(\gamma^{k - 1}))$ and hence $\gamma^k\in B_{k\varepsilon_2}(\widetilde{\mathcal{F}}^u(e))$. It follows that $d(\gamma^k, \widetilde{\mathcal{L}}^u(e)) < k\varepsilon_2 + C_0$. But we already have $d(\gamma^k, \widetilde{\mathcal{L}}^u(e)) > k^{\frac{1}{s}}\varepsilon_1$. Therefore, $k^{\frac{1}{s}}\varepsilon_1 < k\varepsilon_2 + C_0$ for any $k \geq 1$, which is impossible for $k_0$.
	\end{proof}
	
	Now $H'$ is well-defined and uniformly continuous on a dense subset of $N$, so $H'$ can  be extended to a uniformly continuous map on $N$. We still denote it by $H'$.
	
	\begin{proposition}\label{prop: H' = H}
		$H': N \to N$ satisfies the followings.
		\begin{enumerate}
			\item $\Psi\circ H' = H'\circ F$;
			\item $d(H', \text{\rm Id}_N) < +\infty$.
		\end{enumerate}
	\end{proposition}
	
	\begin{proof}[Proof of Proposition \ref{prop: H' = H}]
		(1) First take $x\in \widetilde{\mathcal{F}}^u(\gamma)$, $\gamma\in\Gamma$, then $F(x) \in \widetilde{\mathcal{F}}^u(\Psi(\gamma))$, because $x\gamma^{-1}\in\widetilde{\mathcal{F}}^u(e)$ implies that $F(x)\Psi(\gamma^{-1}) = F(x\gamma^{-1}) \in\widetilde{\mathcal{F}}^u(F(e)) = \widetilde{\mathcal{F}}^u(e)$, and hence $F(x)\in \widetilde{\mathcal{F}}^u(e)\Psi(\gamma) = \widetilde{\mathcal{F}}^u(\Psi(\gamma))$. Consequently,
		\begin{align*}
			\Psi(H'(x)) &= \Psi(H(x\gamma^{-1})\gamma) = \Psi(H(x\gamma^{-1}))\Psi(\gamma)\\
			&= H(F(x\gamma^{-1}))\Psi(\gamma) = H(F(x)\Psi(\gamma^{-1}))\Psi(\gamma) = H'(F(x)).
		\end{align*}
		Now that $\widetilde{\mathcal{F}}^u(\Gamma)$ is dense in $N$ and $H'$ is continuous, the equality holds for all $x\in N$.
		
		(2) For any $x\in \widetilde{\mathcal{F}}^u(\gamma)$, $\gamma\in\Gamma$, we have $d(H'(x), x) = d(H(x\gamma^{-1})\gamma, x) = d(H(x\gamma^{-1}), x\gamma^{-1}) \leq d(H, \text{\rm Id}_N) \leq C_0$. Since $\widetilde{\mathcal{F}}^u(\Gamma)$ is dense in $N$ and $H'$ is continuous, the conclusion holds for all $x\in N$, and hence $d(H', \text{\rm Id}_N) \leq d(H, \text{\rm Id}_N) \leq C_0$.
	\end{proof}
	
	\begin{corollary}\label{cor: H' = H}
		\rm $H' = H$ and consequently $H(x\gamma) = H(x)\gamma$, $\forall x\in N$, $\forall\gamma\in\Gamma$.
	\end{corollary}
	
	\begin{proof}[Proof of Corollary \ref{cor: H' = H}]
		By uniqueness of Lemma \ref{lem: conjugate}, $H' = H$. Therefore, for $x\in\widetilde{\mathcal{F}}^u(\gamma_x)$, $\gamma_x\in\Gamma$, we have $H(x) = H'(x)\in\widetilde{\mathcal{L}}^u(\gamma_x)$, $x\gamma\in\widetilde{\mathcal{F}}^u(\gamma_x\gamma)$ and hence $H(x\gamma) = H'(x\gamma) \in \widetilde{\mathcal{L}}^u(\gamma_x\gamma)$. On the other hand, by Lemma \ref{lem: H and Gamma}, $H(x\gamma) \in\widetilde{\mathcal{L}}^s(H(x)\gamma)$. Thus $H(x\gamma) = \beta(H(x)\gamma, \gamma_x\gamma) = \beta(H(x), \gamma_x)\gamma = H(x)\gamma$. Since $\widetilde{\mathcal{F}}^u(\Gamma)$ is dense in $N$ and $H$ is continuous, the conclusion holds for all $x\in N$.
	\end{proof}
	
	Since $H$ commutes with $\Gamma$, it is projected to a homeomorphism $h$ on $M = N/\Gamma$, satisfying $\Psi\circ h = f\circ h$. Thus $f$ is topologically conjugate to its linear part $\Psi$, and the proof of Theorem \ref{thm: special and conjugacy} is completed.	
\end{proof}

\subsection{A counter example}

To end this section, we construct a nilmanifold endomorphism that is not $u$-ideal and satisfies the following property: there exists $x\in\widetilde{\mathcal{L}}^s(e)$ such that $$\sup\{d(z, \widetilde{\mathcal{L}}^u(e)): z\in\widetilde{\mathcal{L}}^u(x)\} = +\infty.$$

\begin{example}\label{example: holonomy unbounded}
	Consider a 2-step nilpotent Lie algebra 
    $$\mathfrak{n} = \text{span}_\mathbb{R}\{E_{12}, E_{23}, E_{24}, E_{13}, E_{14}\},$$
    where $E_{ij}$ denotes the matrix $(a_{kl})$, $a_{kl} = 1$ for $k = i$, $l = j$, and $a_{kl} = 0$ otherwise. The only nontrivial Lie brackets are $[E_{12}, E_{23}] = E_{13}$, $[E_{12}, E_{24}] = E_{14}$. Define the multiplication by $X\odot Y = X + Y + \frac{1}{2}[X, Y]$, where $[X, Y] := XY - YX$, then $(\mathfrak{n}, \odot)$ is a simply connected 2-step nilpotent Lie group, denoted by $N$. Note that $X^{-1} = -X$, and the commutator $X\odot Y\odot X^{-1}\odot Y^{-1} = (X + Y + \frac{1}{2}[X, Y])\odot(- X - Y + \frac{1}{2}[X, Y]) = [X, Y]$. Define
	$$\psi(E_{12}, E_{23}, E_{24}, E_{13}, E_{14}) = (E_{12}, E_{23}, E_{24}, E_{13}, E_{14})
	\left(\begin{array}{rrrrr} 2&0&0&0&0\\0&2&1&0&0\\0&1&1&0&0\\0&0&0&4&2\\0&0&0&2&2\end{array}\right),$$
	then $\psi\in\text{Aut}(\mathfrak{n})$ is a Lie algebra automorphism and uniquely decide a Lie group automorphism $\Psi\in\text{Aut}(N)$ preserving a lattice $\Gamma = \{(xE_{12})\odot(yE_{23})\odot(zE_{24})\odot(uE_{13})\odot(vE_{14}): x, y, z, u, v\in\mathbb{Z}\}$, hence induces an endomorphism of $N/\Gamma$.
	Notice that the eigenvalues and eigenvectors of $\psi$ are
	\begin{align*}
		&\lambda_1 = 2,& &v_1 = E_{12};\\
		&\lambda_2 = \frac{3 - \sqrt{5}}{2},& &v_2 = -\alpha E_{23} + E_{24};\\
		&\lambda_3 = \frac{3 + \sqrt{5}}{2},& &v_3 = E_{23} + \alpha E_{24};\\
		&\lambda_4 = \lambda_1\lambda_2 = 3 - \sqrt{5},& &v_4 = -\alpha E_{13} + E_{14};\\
		&\lambda_5 = \lambda_1\lambda_3 = 3 + \sqrt{5},& &v_5 = E_{13} + \alpha E_{14},
	\end{align*}
	where $\alpha = \frac{\sqrt{5} - 1}{2}$. Note that $[v_1, v_2] = v_4$, $[v_1, v_3] = v_5$. $\mathfrak{n}^s = \text{span}_\mathbb{R}\{v_2, v_4\}$, $\mathfrak{n}^u = \text{span}_\mathbb{R}\{v_1, v_3, v_5\}$. Take $\mathfrak{n}_{(1)} = \text{span}_\mathbb{R}\{v_1, v_2, v_3\}$, $\mathfrak{n}_{(2)} = \mathfrak{n}_2 = \text{span}_\mathbb{R}\{v_4, v_5\}$, and $P = v_2\in\mathfrak{n}_{(1)}\bigcap\mathfrak{n}^s$.
	
	We claim that $\sup\{d(Q\odot P, \mathfrak{n}^u): Q\in\mathfrak{n}^u\} = +\infty$, or equivalently, $\sup\{d(P\odot[-P, Q], \mathfrak{n}^u): Q\in\mathfrak{n}^u\} = +\infty$. Actually we can show that $\sup\{d(P\odot[-P, tQ], \mathfrak{n}^u): t\in\mathbb{R}\} = +\infty$, where $Q = v_1 \in \mathfrak{n}_{(1)}\bigcap\mathfrak{n}^u$. Denote $R = [-P, Q] = v_4 \in\mathfrak{n}_{(2)}\bigcap\mathfrak{n}^s$, then it suffices to show that $\sup\{d(P \odot tR, \mathfrak{n}^u): t\in\mathbb{R}\} = +\infty$.
    
    Assume for contradiction that there exists $U_t\in\mathfrak{n}^u$ such that $d(P \odot tR, U_t)$ is bounded. Write $U_t = u_1(t)v_1 +u_3(t)v_3 +u_5(t)v_5$. Notice that $d(P \odot tR, tR) = d(P, 0)$, it follows that $d(tR, U_t)$ is also bounded. Then $d((tR)\odot(-U_t), 0) = d(tR - U_t, 0)$ is bounded, hence $q(tR - U_t) \leq Cd(tR - U_t, 0) + C$ is bounded. It follows that $\left\|tR - p_2(U_t)\right\|^{\frac{1}{2}}$ is bounded, which is impossible, because $tR - p_2(U_t) = tv_4 - u_5(t)v_5$.
\end{example}

\section{Density of preimages}\label{Density of preimages}
In this section, $M = N/\Gamma$ is a nilmanifold with a Riemannian metric induced by a right-invariant metric on $N$, where $N$ is a simply connected $s$-step nilpotent Lie group admitting a lattice $\Gamma$.

\begin{definition}\label{def: totally non-invertible}
	A nilmanifold endomorphism $\Psi\in\text{End}(M)$ is totally non-invertible, if the eigenvalues of $\Psi$ are not algebraic units.
\end{definition}

\begin{remark}\label{rmk: totally non-invertible implies hyperbolic}
	If $\Psi$ is totally non-invertible, then $\Psi$ must be hyperbolic, because eigenvalues of $\Psi$ are algebraic integers, while algebraic integers of modulus one are algebraic units. In fact, if $z\in\mathbb{C}$ is an algebraic integer and $|z| = 1$, then $z^{-1} = \bar{z}$ is also an algebraic integer, therefore $z$ is an algebraic unit.
\end{remark}

\begin{remark}\label{rmk: horizontal tni}
	$\Psi$ is totally non-invertible, if and only if the horizontal part $\Psi_1$ is totally non-invertible. In fact, by Lemma \ref{lem: horizontal part}, every eigenvalue of $\Psi$ is the product of some eigenvalues of $\Psi_1$. If $\lambda_1$, $\lambda_2$ are algebraic integers and $\lambda = \lambda_1\lambda_2$ is an algebraic unit, then $\lambda_1$ and $\lambda_2$ are both algebraic unit, because $\lambda_1^{-1} = \lambda_2\lambda^{-1}$, while $\lambda_2$ and $\lambda^{-1}$ are both algebraic integers.
\end{remark}

Recall that $\psi = D_e\Psi\in\text{Aut}(\mathfrak{n})$, and under some basis $\psi$ has a block matrix representation whose diagonal elements are $\psi_i\in\text{Aut}(\mathfrak{n}_i/\mathfrak{n}_{i + 1})$, which is induced by $\psi$. Therefore, $\Psi$ is totally non-invertible actually means that every $\psi_i$ is totally non-invertible, i.e., has no such rational invariant subspace $V$ that the restriction of $\psi_i$ on $V$ has determinant $\pm 1$, hence the induced toral endomorphism of $V/(V\bigcap\Gamma)$ is invertible.

\begin{theorem}\label{thm: exponential density}
	Let $\Psi\in\text{End}(M)$ be a nilmanifold endomorphism. The followings are equivalent.
	\begin{enumerate}
		\item $\Psi$ is totally non-invertible.
		\item There exist $C > 0$, $0 < \mu < 1$, such that $\Psi^{-k}(x)$ is $C\mu^k$-dense in $M$, $\forall x\in M$, $\forall k \geq 1$.
		\item $\bigcup_{k \geq 1}\Psi^{-k}(x)$ is dense in $M$, $\forall x\in M$.
	\end{enumerate}
\end{theorem}

Before the proof of Theorem \ref{thm: exponential density}, we need some preparations.

\begin{lemma}\label{lem: density lemma}
	Let $\pi: \widetilde{M} \to M$ be a fiber bundle with typical fiber $F$, where $\widetilde{M}$, $M$, $F$ are all compact Riemannian manifolds. Then there exists $C \geq 1$ such that for any $\varepsilon_1$, $\varepsilon_2 > 0$ and any subset $S\subseteq \widetilde{M}$ satisfying
	\begin{itemize}
	   \item $\pi(S)$ is $\varepsilon_1$-dense in $M$;
	   \item $S\bigcap \pi^{-1}\pi(x)$ is $\varepsilon_2$-dense in $\pi^{-1}\pi(x)$, $\forall x\in S$,
	\end{itemize}
	$S$ is $C\max\{\varepsilon_1, \varepsilon_2\}$-dense in $\widetilde{M}$.
\end{lemma}

\begin{proof}
	Take a local trivialization $\{(U_i, \varphi_i): 1\leq i \leq N\}$ such that there is a local trivialization $\{(V_i, \psi_i): 1\leq i \leq N\}$ satisfying $\overline{U}_i \subseteq V_i$ and $\varphi_i = \psi_i|_{U_i}$.
	
	Since $\{U_i: 1\leq i \leq N\}$ is an open cover of $M$ and $\pi^{-1}(U_i) = \varphi_i^{-1}(U_i\times F)$ is an open cover of $\widetilde{M}$, there exists $0 < \varepsilon_0 \leq \min\{\text{diam}(M), \text{diam}(\widetilde{M})\}$ such that every $\varepsilon_0$-ball in $M$ lies in some $U_i$, and that every $\varepsilon_0$-ball in $\widetilde{M}$ lies in some $\pi^{-1}(U_i)$.
	
	Recall that two distances $d$ and $d'$ on a topological space are called equivalent, if there exists $C > 1$ such that $C^{-1}d(x, y) \leq d'(x, y) \leq Cd(x, y)$, which means every $\varepsilon$-ball under $d'$ contains some $C^{-1}\varepsilon$-ball under $d$ and vice versa.
	
	Define a distance on $\overline{U}_i\times F$ by $d((u_1, v_1), (u_2, v_2)) = \max\{d(u_1, u_2), d(v_1, v_2)\}$, for $(u_1, v_1), (u_2, v_2)\in \overline{U}_i\times F$. Such a distance is equivalent with $d'((u_1, v_1), (u_2, v_2)) := (d(u_1, u_2)^2 + d(v_1, v_2)^2)^{\frac{1}{2}}$, which is induced by the product Riemannian metric.
	
	Both of $d$ and $d'$ induce some distance on $\pi^{-1}(\overline{U}_i)$ by $\varphi_i$, and by compactness, the latter is equivalent with the original distance on $\pi^{-1}(\overline{U}_i)$ induced by the inclusion $\pi^{-1}(\overline{U_i}) \to \widetilde{M}$. Consequently, there are constants $C_i \geq 1$, $1\leq i \leq N$, such that every $\varepsilon$-ball in $\pi^{-1}(U_i)$ contains some $\varphi_i^{-1}(B_U\times B_F)$, where $B_U$ and $B_F$ are $C_i^{-1}\varepsilon$-balls in $U_i$ and $F$ respectively.
		
	Now take $C = \varepsilon_0^{-1}\text{diam}(\widetilde{M})\max\{C_i^2: 1\leq i \leq N\}$. Assume that $\max\{C_i^2: 1\leq i \leq N\}\max\{\varepsilon_1, \varepsilon_2\} \leq \varepsilon_0$, otherwise $C\max\{\varepsilon_1, \varepsilon_2\} > \text{diam}(\widetilde{M})$ and the conclusion is obvious. Take any $C\max\{\varepsilon_1, \varepsilon_2\}$-ball $B_0$ in $\widetilde{M}$, one needs to show that $S\bigcap B_0 \not= \varnothing$. Notice that there is a ball $B\subset B_0$ with radius $r = \min\{C\max\{\varepsilon_1, \varepsilon_2\}, \varepsilon_0\}$, such that $B$ lies in some $\pi^{-1}(U_i)$. Then $B$ contains some $\varphi_i^{-1}(B_U\times B_F)$, where $B_U$ and $B_F$ are $C_i^{-1}r$-balls in $U_i$ and $F$ respectively. Thus $\pi(B)$ contains $B_U$, which contains some point $z\in \pi(S)$, because $C_i^{-1}r \geq \varepsilon_1$. Moreover, $B \bigcap \pi^{-1}(z)$ contains $\varphi_i^{-1}(\{z\}\times B_F)$, which contains a $C_i^{-2}r$-ball in $\pi^{-1}(z)$, and thus contains some point $x\in S$, because $C_i^{-2}r \geq \varepsilon_2$.
\end{proof}

\begin{definition}\label{def: irreducible}
	A toral endomorphism is {\it irreducible}, if its characteristic polynomial is $\mathbb{Q}$-irreducible.
\end{definition}

\begin{lemma}\label{lem: irreducible density}
	{\rm (\cite{An2023}, Proposition 2.10)} Assume that $A\in GL(n,\mathbb{R})\bigcap M(n, \mathbb{Z})$, $|\det A| > 1$. If $A$ is irreducible, then there exists $C > 0$, such that $A^{-k}\mathbb{Z}^n$ is $C|\det A|^{-\frac{k}{n}}$-dense in $\mathbb{R}^n$, $\forall k \geq 1$.
\end{lemma}

\begin{corollary}\label{cor: irreducible density}
	Let $V$ be an $n$-dimensional vector space admitting a lattice $\Gamma$, $A \in GL(V)$ and $A\Gamma\subsetneq\Gamma$. If $A\in\text{End}(V/\Gamma)$ is irreducible, then there exists $C > 0$ such that $A^{-k}\Gamma$ is $C|\det A|^{-\frac{k}{n}}$-dense in $V$, $\forall k \geq 1$.
\end{corollary}

\begin{proof}
	Take a $\mathbb{Z}$-basis of $\Gamma$, the conclusion follows from Lemma \ref{lem: irreducible density} immediately.
\end{proof}

\begin{lemma}\label{lem: rcf}
	{\rm (\cite{Hartwig1996}, Rational Canonical Form)} Assume that $A\in M(n, \mathbb{Q})$. Then there exists $Q\in GL(n, \mathbb{Q})$ such that $Q^{-1}AQ = \text{diag}\{L(g_1), \cdots, L(g_k)\}$, where $g_1$, $\cdots$, $g_k \in \mathbb{Q}[\lambda]$ are monic polynomials satisfying $g_j|g_{j + 1}$, $1\leq j \leq k - 1$, and $L(g_j)$ is the companion matrix of $\ g_j$.
\end{lemma}

\begin{corollary}\label{cor: rcf}
	Assume that $A\in GL(n, \mathbb{R})\bigcap M(n, \mathbb{Z})$, the characteristic polynomial of $A$ is $f(\lambda) = g^s(\lambda)$, where $s\in\mathbb{Z}^+$ and $g(\lambda)\in\mathbb{Z}[\lambda]$ is $\mathbb{Q}$-irreducible, and the minimal polynomial of $A$ is $g(\lambda)$. Then there exists an $A$-invariant rational decomposition $\mathbb{R}^n = V = \bigoplus_{j = 1}^sV_j$, such that the characteristic polynomial of $A|_{V_j}$ is $g(\lambda)$, $1\leq j \leq s$. Further, there exists $C > 0$ such that $A^{-k}\mathbb{Z}^n$ is $C|\det A|^{-\frac{k}{n}}$-dense in $\mathbb{R}^n$.
\end{corollary}

\begin{proof}
	By Lemma \ref{lem: rcf}, there exists an $A$-invariant rational decomposition $\mathbb{R}^n = V = \bigoplus_{j = 1}^kV_j$, such that the characteristic polynomial of $A_j := A|_{V_j}$ is $g_j$, and $g_j|g_{j + 1}$, $1\leq j \leq k - 1$. Notice that the minimal polynomial of $L(g_j)$ is $g_j$, hence the minimal polynomial of $A$ is $g_k$. Therefore $g_k = g$, and because of the $\mathbb{Q}$-irreducibility of $g$, we have $g_j = g$, $1\leq j \leq k$, hence $k = s$.
	
	Denote $\Gamma_j = \mathbb{Z}^n\bigcap V_j$, which is a lattice of $V_j$, and $A_j\Gamma_j \subsetneq\Gamma_j$, since $|\det A_j|  = |\det A|^{\frac{1}{s}} > 1$. The characteristic polynomial of $A_j$ is $g(\lambda)$, which is $\mathbb{Q}$-irreducible. Denote $l = \deg g$, then $sl = n$. By Corollary \ref{cor: irreducible density}, $A_j^{-k}\Gamma_j$ is $C_j|\det A_j|^{-\frac{k}{l}}$-dense in $V_j$ for some constant $C_j > 0$. Thus $A^{-k}\mathbb{Z}^n \supseteq \bigoplus_{j = 1}^sA_j^{-k}\Gamma_j$ is $C|\det A|^{-\frac{k}{n}}$-dense in $\mathbb{R}^n$ for some $C > 0$.
\end{proof}

\begin{theorem}\label{thm: density torus}
	Assume that $A\in GL(n,\mathbb{R})\bigcap M(n, \mathbb{Z})$. Then $A$ is totally non-invertible if and only if there are constants $C > 0$, $0< \mu < 1$, such that $A^{-k}\mathbb{Z}^n$ is $C\mu^k$-dense in $\mathbb{R}^n$, $\forall k \geq 1$.
\end{theorem}

\begin{proof}
	To show the necessity, assume for contradiction that there is an algebraic unit in the eigenvalues of $A$, then the characteristic polynomial of $A$ has a factorization $f = gh$, where $g, h\in\mathbb{Z}[x]$, $(g, h) = 1$, and the constant term of $g$ is $\pm 1$. Hence there exists $Q\in GL(n, \mathbb{Q})$ such that $Q^{-1}AQ = B = \text{diag}\{A_g, A_h\}$, the characteristic polynomials of $A_g$, $A_h$ are $g$, $h$ respectively, and $A_g \in GL(n_g, \mathbb{Z})$, $n_g = \dim\ker g(A)$. Assume $Q^{-1}\mathbb{Z}^n \subseteq q^{-1}\mathbb{Z}^n$ for some $q\in\mathbb{Z}^+$, then
	$$A^{-k}\mathbb{Z}^n = QB^{-k}Q^{-1}\mathbb{Z}^n \subseteq q^{-1}Q(A_g^{-k}\mathbb{Z}^{n_g}\bigoplus \mathbb{R}^{n - n_g}) \subseteq q^{-1}Q(\mathbb{Z}^{n_g}\bigoplus\mathbb{R}^{n - n_g}),$$ which cannot be arbitrarily dense in $\mathbb{R}^n$.
	
	To show the sufficiency, let $f$ be the characteristic polynomial of $A$. First consider the case when $f(\lambda) = g^s(\lambda)$, where $g\in\mathbb{Z}[x]$ is $\mathbb{Q}$-irreducible, $s\in\mathbb{Z}^+$. Assume that the minimal polynomial of $A$ is $m(\lambda) = g^t(\lambda)$, $1\leq t \leq s$.
	Denote $V_j = \ker g^j(A)$, then one has the $A$-invariant rational filtration
	$$\{0\} = V_0 \subsetneq V_1 \subsetneq \cdots \subsetneq V_t = V = \mathbb{R}^n.$$
	
	Thus $\Gamma_j := \mathbb{Z}^n\bigcap V_j$ is a lattice in $V_j$, and $\Gamma :=\Gamma_t = \mathbb{Z}^n$. $A\Gamma_j \subseteq \Gamma_j$. Consider the fiber bundle $\pi_j: V/V_j\Gamma \to V/V_{j + 1}\Gamma$. It is a $V_{j + 1}/V_j\Gamma_{j + 1}$-principal bundle.
	
	One can prove by induction that, there are constants $C_j > 0$, $0 < \mu_j < 1$ such that $A_j^{-k}(x)$ is $C_j\mu_j^k$-dense in $V/V_j\Gamma$, $\forall x\in V/V_j\Gamma$, $0\leq j \leq t$, where $A_j: V/V_j\Gamma \to V/V_j\Gamma$ is induced by $A$. The basic case is $j = t$ and there is nothing to prove. The target is $j = 0$, once $A_0^{-k}(x)$ is $C_0\mu_0^k$-dense in $V/V_0\Gamma = V/\Gamma = \mathbb{R}^n/\mathbb{Z}^n = \mathbb{T}^n$, the proof is completed since $\pi: \mathbb{R}^n \to\mathbb{T}^n$ is locally isometric.
	
	Assume that the conclusion holds for $j + 1$, where $j \leq t - 1$. Consider the fiber bundle $\pi_j: V/V_j\Gamma \to V/V_{j + 1}\Gamma$, one has the following commuting diagram
	$$\xymatrix{V/V_j\Gamma \ar[d]_{\pi_j}\ar[r]^{A_j} & V/V_j\Gamma \ar[d]^{\pi_j}\\ V/V_{j + 1}\Gamma \ar[r]_{A_{j + 1}} & V/V_{j + 1}\Gamma}$$
	
	For any fixed $x\in V/V_j\Gamma$, by inductive hypothesis, $A_{j + 1}^{-k}\pi_j(x)$ is $C_{j + 1}\mu_{j + 1}^k$-dense in $V/V_{j + 1}\Gamma$. Take any $\bar{z}$ in it, then $A_j^k$ maps the fiber $\pi_j^{-1}(\bar{z})$ to the fiber $\pi_j^{-1}\pi_j(x)$. Since the fiber is a translation of $V_{j + 1}/V_j\Gamma_{j + 1}$, the restriction of $A_j^k$ on this fiber is the composition of $A^k_{j, j + 1}: V_{j + 1}/V_j\Gamma_{j + 1} \to V_{j + 1}/V_j\Gamma_{j + 1}$ and some translations. Notice that the characteristic polynomial of $A_{j, j + 1}$ is a power of $g(\lambda)$, and the minimal polynomial of $A_{j, j + 1}$ is $g(\lambda)$, which is $\mathbb{Q}$-irreducible, and $|\det A_{j, j + 1}| > 1$, since $A$ is totally non-invertible. Now apply Corollary \ref{cor: rcf} to $A_{j, j + 1}$, we have that $A_j^{-k}(x)\bigcap \pi_j^{-1}(\bar{z})$ is $C_{j, j + 1}\mu_{j, j + 1}^k$-dense in the fiber for some $C_{j , j + 1} > 0$ and $0 < \mu_{j, j + 1} < 1$. By Lemma \ref{lem: density lemma}, $A_j^{-k}(x)$ is $C_j\mu_j^k$-dense in $V/V_j\Gamma$, where $C_j $ depends on $C_{j + 1}$ and $C_{j, j + 1}$, $\mu_j = \max\{\mu_{j + 1}, \mu_{j, j + 1}\}$, both of them are independent of $k$ and $x$. The induction is completed.
	
	Now consider the general case. The characteristic polynomial of $A$ has a factorization $f(\lambda) = \Pi_{i = 1}^r g_i^{s_i}(\lambda)$, and $\mathbb{R}^n$ has an $A$-invariant rational decomposition $\mathbb{R}^n = \bigoplus_{i = 1}^r\ker g_i^{s_i}(A)$. Again denote $V_i = \ker g_i^{s_i}(A)$, $A_i = A|_{V_i}$, $\Gamma_i = V_i \bigcap\mathbb{Z}^n$, and apply the above argument to $(V_i, A_i)$. Then there are constants $C_i > 0$, $0 < \mu_i < 1$ such that $A_i^{-k}\Gamma_i$ is $C_i\mu_i^k$-dense in $V_i$. Therefore $A^{-k}\mathbb{Z}^n \supseteq A^{-k}(\bigoplus_{i = 1}^r\Gamma_i) = \bigoplus_{i = 1}^r A_i^{-k}\Gamma_i$ is $C\mu^k$-dense in $\mathbb{R}^n$, where $C$ depends on $C_1$, $\cdots$, $C_r$ and $\mu = \max\{\mu_i: 1\leq i \leq r\}$.
\end{proof}

Finally, we can handle the nilmanifold case.

\begin{proof}[Proof of Theorem \ref{thm: exponential density}]
	Denote $M_i = N/N_{i + 1}\Gamma$, $\Psi_i\in\text{End}(M_i)$ is induced by $\Psi\in\text{End}(M)$, $\pi_i: M_i \to M_{i - 1}$ is the canonical projection, then $\pi_i: M_i \to M_{i - 1}$ is an $N_i/N_{i + 1}\Gamma_i = \mathbb{T}^{d_i}$-principal bundle, and we have the following commuting diagram, $1\leq i \leq s$.
	$$\xymatrix{
		M_i \ar[r]^{\Psi_i} \ar[d]_{\pi_i} & M_i \ar[d]^{\pi_i}\\
		M_{i - 1} \ar[r]^{\Psi_{i - 1}} & M_{i - 1}\\
	}$$
	Recall that $\psi = D_e\Psi\in\text{Aut}(\mathfrak{n})$, $\psi_i \in\text{Aut}(\mathfrak{n}_i/\mathfrak{n}_{i + 1})$.
	
	$(2)\Rightarrow(3)$ is obvious.
	
	$(3)\Rightarrow(1)$ can be proved as follows. If there is an algebraic unit in the eigenvalues of $\psi$, then the same property must hold for some $\psi_i$. Since the fibers are homeomorphic to tori, from the torus case, the preimages of zero in $N_i/N_{i + 1}\Gamma_i$ lie in a closed proper submanifold of it. Therefore, in the diagram above, the preimages of $eN_{i + 1}\Gamma$ in $M_i$ lie in some compact proper subbundle of $M_i$ over $M_{i - 1}$, which cannot be dense in $M_i$, and hence the conclusion does not hold for $M$.
	
	To show $(1)\Rightarrow(2)$, one can prove by induction that there are constants $C_i > 0$, $0 < \mu_i < 1$, such that $\forall x\in M_i$, $\Psi_i^{-k}(x)$ is $C_i\mu_i^k$-dense in $M_i$. The base case is $i = 1$. Since $M_0$ is a single point, there is nothing to prove. The target is the case when $i = s$. Note that $M_s = M$, $\Psi_s = \Psi$.
	
	Assume $i \geq 2$ and that the conclusion holds for $i - 1$. $\Psi_i \in \text{End}(M_i)$ maps the fiber $\pi_i^{-1}\pi_i(x)$ to the fiber $\pi_i^{-1}\Psi_{i - 1}\pi_i(x)$, $\forall x\in M_i$. Since every fiber is a left translation of $N_i/N_{i + 1}\Gamma_i$, the restriction of $\Psi_i$ on every fiber is a composition of an endomorphism on torus, $\Psi_{i, i + 1} : N_i/N_{i + 1}\Gamma_i \to N_i/N_{i + 1}\Gamma_i$, and some left translations.
	
	By inductive hypothesis, there are constants $C_{i - 1} > 0$, $0 < \mu_{i - 1} < 1$ such that $\Psi_{i - 1}^{-k}(\pi_i(x))$ is $C_{i - 1}\mu_{i - 1}^k$-dense in $M_{i - 1}$. For any $\bar{z} \in \Psi_{i - 1}^{-k}(\pi_i(x))$, $\Psi_i^k$ maps the fiber $\pi_i^{-1}(\bar{z})$ to the fiber $\pi_i^{-1}\pi_i(x)$. Moreover, $D_e\Psi_{i, i + 1} = \psi_i: \mathfrak{n}_i/\mathfrak{n}_{i + 1} \to \mathfrak{n}_i/\mathfrak{n}_{i + 1}$. Since the characteristic polynomial $f_i(\lambda)$ of $\psi_i$ is a factor of $f(\lambda)$, $f_i(\lambda)$ satisfies the condition in Theorem \ref{thm: density torus}. Therefore, there are constants $C > 0$, $0 < \mu < 1$, independent of $k$, $x$, $\bar{z}$, such that $\Psi_i^{-k}(x) \bigcap \pi_i^{-1}\pi_i(\bar{z})$ is $C\mu^k$-dense in the fiber $\pi_i^{-1}\pi_i(\bar{z})$. By Lemma \ref{lem: density lemma}, there are constants $C_i = C_i(C_{i - 1}, C) > 0$, $0 < \mu_i =\max\{\mu_{i - 1}, \mu\} < 1$, such that $\Psi_i^{-k}(x)$ is $C_i\mu_i^k$-dense in $M_i$.
\end{proof}

\section{Rigidity of conjugacy and stable Lyapunov exponents}\label{Rigidity of conjugacy and stable Lyapunov exponents}

In this section, $M = N/\Gamma$ is a nilmanifold with a Riemannian metric induced by a right-invariant Riemannian metric on $N$, $f$ is an Anosov map on $M$ with one-dimensional stable bundle, $F$ is a lift of $f$, $\Psi$ is the linear part of $F$, $H$ is the conjugacy between $F$ and $\Psi$ constructed in Lemma \ref{lem: conjugate}. Under some conditions, we will prove the equivalence between the existence of conjugacy and the same stable Lyapunov exponents at periodic points.

\subsection{Conjugacy implies the same stable Lyapunov exponents}

In this subsection, with the help of exponential density of preimage set, we prove that the existence of conjugacy between $f$ and $\Psi$ implies that the periodic stable Lyapunov exponents of $f$ are the same as $\Psi$. For convenience, we restate Theorem \ref{intro thm: conjuacy implies pdc} as follows.

\begin{theorem}\label{thm: conjugacy implies pdc}
	If $f$ is topologically conjugate to $\Psi$, which is totally non-invertible, then $\lambda^s(p, f) = \lambda^s(q, f)$, for all $p, q\in\text{Per}(f)$.
\end{theorem}

\begin{proof}
	By Proposition \ref{prop: unique H}, $H$ commutes with $\Gamma$ and is projected to a conjugacy $h$ between $f$ and $\Psi$. Moreover, by Lemma \ref{lem: Holder}, $H$ is bi-$\alpha$-H\"older continuous for some $0 < \alpha < 1$. Hence $h$ is also bi-$\alpha$-H\"older continuous. 
	
	Assume for contradiction that $\mu_- < \mu_+$, then we can fix $\delta > 0$ such that $(1 + \delta)^4 < \left(\frac{\mu_+}{\mu_-}\right)^\alpha$. By Lemma \ref{lem: adapted metric}, there is an adapted metric for $f$, denoted by $\left|\cdot\right|$, such that $\mu_-(1 + \delta)^{-1} < |Df|_{E^s(x)}| < \mu_+(1 + \delta)$, $\forall x\in M$. We also use $\left|\cdot\right|$ to denote the length of a curve under the adapted metric. Note that $\left\|\cdot\right\|$ is used to denote the original metric on $M$ induced by a right-invariant metric on $N$, as well as the length of a curve under the metric.
	
	\begin{claim*}
		There are constants $0 < \mu < 1$, $C > 1$ such that the followings hold.
		\begin{enumerate}
			\item $C^{-1}\left\|\cdot\right\| \leq \left|\cdot\right| \leq C\left\|\cdot\right\|$;
			\item For $x, y\in N$, $d(x, y)\leq 1$ implies that $d(H(x), H(y)) \leq Cd(x, y)^\alpha$;
			\item For $x, y\in N$, $d(x, y)\leq 1$ implies that $d_{\widetilde{\mathcal{F}}^s}(x, \beta_{\widetilde{\mathcal{F}}}(x, y)) + d_{\widetilde{\mathcal{F}}^u}(\beta_{\widetilde{\mathcal{F}}}(x, y), y) \leq Cd(x, y)$;
			\item For $x, y\in N$, $d(x, y)\leq 1$ implies that $d_{\widetilde{\mathcal{L}}^s}(x, \beta(x, y)) + d_{\widetilde{\mathcal{L}}^u}(\beta(x, y), y) \leq Cd(x, y)$;
			\item For any $\omega \in M$, $f^{-k}(\omega)$ is $C\mu^{k\alpha}$-dense in $M$.
		\end{enumerate}
	\end{claim*}
	
	\begin{proof}[Proof of Claim]
		For 1, notice that $\left|\cdot\right|$ is equivalent with $\left\|\cdot\right\|$.
		
		For 2, notice that $d$ is right-invariant, $H$ is $\alpha$-H\"older continuous.
		
		For 3, notice that $d$, $\widetilde{\mathcal{F}}^s$, $\widetilde{\mathcal{F}}^u$ are right-$\Gamma$-invariant, and $\mathfrak{n} = \widetilde{E}^s(e)\oplus\widetilde{E}^u(e)$.
		
		For 4, notice that $d$, $\widetilde{\mathcal{L}}^s$, $\widetilde{\mathcal{L}}^u$ are right-invariant, and $\mathfrak{n} = \mathfrak{n}^s\oplus\mathfrak{n}^u$.
		
		For 5, notice that $\Psi$ is totally non-invertible, $f$ and $\Psi$ are conjugate via a bi-$\alpha$-H\"older continuous homeomorphism, we can apply Theorem \ref{thm: exponential density}.
	\end{proof}
	
	Return to the proof of theorem \ref{thm: conjugacy implies pdc}. There are periodic points $p$, $q$ such that $\mu^s(p, f) < \mu_-(1 + \delta)$, $\mu^s(q, f) > \mu_+(1 + \delta)^{-1}$. By considering an iteration of $f$ if necessary, assume that $p$, $q$ are fixed points, then $\mu^s(p, f) = \|Df|_{E^s(p)}\|$, $\mu^s(q, f) = \|Df|_{E^s(q)}\|$.
	
	For $\varepsilon > 0$, there exist a positive integer $k(\varepsilon)$ and some point $x(\varepsilon) \in B(q, \varepsilon)$ such that $f^{k(\varepsilon)}(x(\varepsilon)) = p$. Take $k(\varepsilon)$ to be the smallest positive integer satisfying this property, then
	$$k(\varepsilon) \leq \frac{\ln C - \ln\varepsilon}{-\alpha\ln\mu} + 1.$$
	
	The following claim is obvious.
	
	\begin{claim*}
		There are constants $0 < \eta'_0 < \eta_0 < 1$ such that for any $\widetilde{p}\in\pi^{-1}(p)$, $\widetilde{q}\in\pi^{-1}(q)$, the followings hold.
		\begin{enumerate}
			\item $\pi^{-1}B(p, \eta_0) = \bigcup_{\gamma\in\Gamma}B(\widetilde{p}, \eta_0)\gamma$, which is a disjoint union;
			\item $\pi: B(\widetilde{p}, \eta_0) \to B(p, \eta_0)$ is an isometry;
			\item For $x\in B(\widetilde{p}, \eta_0)$, $\left\|DF|_{\widetilde{E}^s(x)}\right\| < \mu^s(p, f)(1 + \delta) < \mu_-(1 + \delta)^2$;
			\item $\pi: HB(\widetilde{p}, \eta_0) \to hB(p, \eta_0)$ is an isometry, and $B(H(\widetilde{p}), \eta'_0) \subseteq HB(\widetilde{p}, \eta_0)$;
			\item $\pi^{-1}B(q, \eta_0) = \bigcup_{\gamma\in\Gamma}B(\widetilde{q}, \eta_0)\gamma$, which is a disjoint union;
			\item $\pi: B(\widetilde{q}, \eta_0) \to B(q, \eta_0)$ is an isometry;
			\item For $x\in B(\widetilde{q}, \eta_0)$, $\left\|DF|_{\widetilde{E}^s(x)}\right\| > \mu^s(q, f)(1 + \delta)^{-1} > \mu_+(1 + \delta)^{-2}$;
			\item $\pi: HB(\widetilde{q}, \eta_0) \to hB(q, \eta_0)$ is an isometry, and $B(H(\widetilde{q}), \eta'_0) \subseteq HB(\widetilde{q}, \eta_0)$.
		\end{enumerate}
	\end{claim*}
	
	Now fix some $\widetilde{q}\in\pi^{-1}(q)$. Choose $\eta > 0$ such that $2C^2\eta^\alpha < \eta'_0$.
	
	When $C\varepsilon < \eta$, let $\widetilde{x}(\varepsilon)$ be the unique point in $\pi^{-1}(x(\varepsilon))\bigcap B(\widetilde{q}, \eta_0)$, $\widetilde{y}(\varepsilon) = \beta_{\widetilde{\mathcal{F}}}(\widetilde{x}(\varepsilon), \widetilde{q})$. Then we have $d(\widetilde{y}(\varepsilon), \widetilde{q}) < Cd(\widetilde{x}(\varepsilon), \widetilde{q}) < C\varepsilon < \eta < \eta_0$, so $y(\varepsilon) := \pi(\widetilde{y}(\varepsilon))$ lies in $B(q, \eta_0)$.
	
	Let $\widetilde{I}(\varepsilon)$ be a curve lying in $\widetilde{\mathcal{F}}^s(\widetilde{x}(\varepsilon))$ with length $\eta$ and containing $\widetilde{x}(\varepsilon)$ and $\widetilde{y}(\varepsilon)$. Such curve exists, because $d_{\widetilde{\mathcal{F}}^s}(\widetilde{x}(\varepsilon), \widetilde{y}(\varepsilon)) < Cd(\widetilde{x}(\varepsilon), \widetilde{q}) < C\varepsilon < \eta$. Since $d(\widetilde{y}(\varepsilon), \widetilde{q}) + \widetilde{I}(\varepsilon) < C\varepsilon + \eta < 2\eta < \eta_0$, the curve lies in $B(\widetilde{q}, \eta_0)$, hence $I(\varepsilon) := \pi(\widetilde{I}(\varepsilon))\subseteq B(q, \eta_0)$, which is a curve lying in $\mathcal{F}^s(x(\varepsilon))$ with length $\eta$, containing $x(\varepsilon)$ and $y(\varepsilon)$.
	
	Let $N(\varepsilon)$ be the minimal positive integer such that $f^{N(\varepsilon)}(I(\varepsilon))$ is not contained in $B(q, \eta_0)$. To estimate $N(\varepsilon)$, consider $H(B(\widetilde{q}, \eta_0))$, which is a neighborhood of $H(\widetilde{q})$, and is projected to $h(B(q, \eta_0))$ isometrically. Notice that $H(\widetilde{I}(\varepsilon))$ is a curve lying in $\widetilde{\mathcal{L}}^s(H(\widetilde{x}(\varepsilon)))$, containing $H(\widetilde{x}(\varepsilon))$ and $H(\widetilde{y}(\varepsilon))$. Denote the distance between the endpoints of a curve $r$ by $d(r)$. We have
	\begin{align*}
		&d(H(\widetilde{q}), H(\widetilde{y}(\varepsilon))) \leq Cd(\widetilde{q}, \widetilde{y}(\varepsilon))^\alpha \leq C(C\varepsilon)^\alpha;\\
		&d_{\widetilde{\mathcal{L}}^u}(\Psi^n(H(\widetilde{q})), \Psi^n(H(\widetilde{y}(\varepsilon)))) \leq C^2(C\varepsilon)^\alpha\mu^u_+(\Psi)^n;\\
		&\|H(\widetilde{I}(\varepsilon))\| \leq Cd(H(\widetilde{I}(\varepsilon)))\leq C^2d(\widetilde{I}(\varepsilon))^\alpha \leq C^2\|\widetilde{I}(\varepsilon)\|^\alpha = C^2\eta^\alpha;\\
		&\|\Psi^n(H(\widetilde{I}(\varepsilon)))\| = \mu^s(\Psi)^n\|H(\widetilde{I}(\varepsilon))\| \leq C^2\eta^\alpha\mu^s(\Psi)^n.
	\end{align*}
	
	As long as $C^2(C\varepsilon)^\alpha\mu^u_+(\Psi)^n + C^2\eta^\alpha\mu^s(\Psi)^n < \eta'_0$, $\Psi^n(H(\widetilde{I}(\varepsilon)))$ must lie in $B(\Psi^n(H(\widetilde{q})), \eta'_0)$, which is projected to $B(h(q), \eta'_0)$ isometrically, since $h(q)$ is a fixed point of $\Psi$. For such $n$, $F^n(\widetilde{I}(\varepsilon))$ must lie in $B(\widetilde{q}, \eta_0)$. Therefore,
	$$N(\varepsilon) \geq \frac{\ln\eta'_0 - (2 + \alpha)\ln C - \ln 2 - \alpha\ln\varepsilon}{\ln\mu^u_+(\Psi)} - 1.$$
	Consequently,
	$$\liminf_{\varepsilon \to 0}\frac{N(\varepsilon)}{k(\varepsilon)} \geq \alpha.$$
	
	Since $\delta$ is fixed so that $(1 + \delta)^4 < \left(\frac{\mu_+}{\mu_-}\right)^\alpha$, and $k(\varepsilon) \to +\infty$ as $\varepsilon \to 0$, we can fix $\varepsilon$ sufficiently small such that $C\varepsilon < \eta$, $\|Df^{k(\varepsilon)}|_{E^s}\| < 1$, and
	$$\left(\left(\frac{\mu_+}{\mu_-}\right)^{\frac{N(\varepsilon)}{k(\varepsilon)}}\frac{1}{(1 + \delta)^4}\right)^{k(\varepsilon)} > \frac{C^6}{\eta}.$$
	
	Now fix $\widetilde{p} = F^{k(\varepsilon)}(\widetilde{x}(\varepsilon)) \in \pi^{-1}(p)$.
	
	Since $f^{k(\varepsilon)}(x(\varepsilon)) = p$, one has $f^{k(\varepsilon)}(I(\varepsilon)) \subseteq \mathcal{F}^s(p)$, $F^{k(\varepsilon)}(\widetilde{I}(\varepsilon)) \subseteq \widetilde{\mathcal{F}}^s(\widetilde{p})$. $\left\|f^{k(\varepsilon)}(I(\varepsilon))\right\| < \left\|I(\varepsilon)\right\|$, so the curve is contained in $B(p, \eta_0)$. Similarly, $F^{k(\varepsilon)}(\widetilde{I}(\varepsilon))$ is contained in $B(\widetilde{p}, \eta_0)$ and is projected to $f^{k(\varepsilon)}(I(\varepsilon))$ isometrically.
	
	On the other hand, $p$ is a fixed point of $f$, thus $B(F^{k(\varepsilon)}(\widetilde{p}), \eta_0)$ is also projected to $B(p, \eta_0)$ isometrically. Let $F^{k(\varepsilon)}(\widetilde{J}(\varepsilon))$ be the lift of $f^{k(\varepsilon)}(I(\varepsilon))$ in $B(F^{k(\varepsilon)}(\widetilde{p}), \eta_0)$. Then $F^{k(\varepsilon)}(\widetilde{J}(\varepsilon))$ is the image of some right-$\Gamma$-translation of $F^{k(\varepsilon)}(\widetilde{I}(\varepsilon))$, since they have the same projection.
	
	Clearly $\widetilde{J}(\varepsilon)$ contains $\widetilde{p}$. To show that $\widetilde{J}(\varepsilon)$ is contained in $B(\widetilde{p}, \eta_0)$, notice that $H^{-1}$ commutes with $\Gamma$, therefore $\Psi^{k(\varepsilon)}(H(\widetilde{J}(\varepsilon)))$ is the image of some right-$\Gamma$-translation of $\Psi^{k(\varepsilon)}(H(\widetilde{I}(\varepsilon)))$. Hence $H(\widetilde{J}(\varepsilon))$ is also the image of some right translation of $H(\widetilde{I}(\varepsilon))$, and thus $\|H(\widetilde{J}(\varepsilon))\| = \|H(\widetilde{I}(\varepsilon))\| \leq C\eta^\alpha < \eta'_0$. Consequently, $H(\widetilde{J}(\varepsilon))$ is contained in $B(H(\widetilde{p}), \eta'_0)$, $\widetilde{J}(\varepsilon)$ is contained in $B(\widetilde{p}, \eta_0)$ and is projected isometrically to $J(\varepsilon) :=\pi (\widetilde{J}(\varepsilon))$, which is contained in $B(p, \eta)$, and thus $H(\widetilde{J}(\varepsilon))$ is projected isometrically to $h(J(\varepsilon))$.
	
	Therefore, $J(\varepsilon) \subseteq \mathcal{F}^s(p)$, $f^{k(\varepsilon)}(J(\varepsilon)) = f^{k(\varepsilon)}(I(\varepsilon))$, $\|h(J(\varepsilon))\| = \|H(\widetilde{J}(\varepsilon))\| = \|H(\widetilde{I}(\varepsilon))\| = \|h(I(\varepsilon))\|$.
	
	From the above discussion, we have
	\begin{align*}
		\left\|f^{N(\varepsilon)}(I(\varepsilon))\right\| &\geq (\mu_+(1 + \delta)^{-2})^{N(\varepsilon)}\left\|I(\varepsilon)\right\|;\\
		\left|f^{k(\varepsilon)}(I(\varepsilon))\right| &\geq (\mu_-(1 + \delta)^{-1})^{k(\varepsilon) - N(\varepsilon)}\left|f^{N(\varepsilon)}(I(\varepsilon))\right|;\\
		\left\|f^{k(\varepsilon)}(J(\varepsilon))\right\| &\leq (\mu_-(1 + \delta)^2)^{k(\varepsilon)}\left\|J(\varepsilon)\right\|;\\
		\left\|J(\varepsilon)\right\| &\leq Cd(J(\varepsilon)) \leq C^2d(h(J(\varepsilon)))^\alpha = C^2C^\alpha\left\|h(J(\varepsilon))\right\|^\alpha\\
        &= C^2C^\alpha\left\|h(I(\varepsilon))\right\|^\alpha \leq C^2C^\alpha(C\eta^\alpha)^\alpha.
	\end{align*}
	
	Therefore,
	\begin{align*}
		1 = \frac{\left\|f^{k(\varepsilon)}(I(\varepsilon))\right\|}{\left\|f^{k(\varepsilon)}(J(\varepsilon))\right\|}
		&\geq C^{-2}\frac{(\mu_-(1 + \delta)^{-1})^{k(\varepsilon) - N(\varepsilon)}(\mu_+(1 + \delta)^{-2})^{N(\varepsilon)}\left\|I(\varepsilon)\right\|}{(\mu_-(1 + \delta)^2)^{k(\varepsilon)}\left\|J(\varepsilon)\right\|}\\
		& \geq C^{-2}\left(\frac{\mu_+}{\mu_-}\right)^{N(\varepsilon)}\frac{1}{(1 + \delta)^{3k(\varepsilon) + N(\varepsilon)}}\frac{\left\|I(\varepsilon)\right\|}{\left\|J(\varepsilon)\right\|}\\
		& \geq \left(\frac{\mu_+}{\mu_-}\right)^{N(\varepsilon)}\frac{1}{(1 + \delta)^{4k(\varepsilon)}}\frac{\eta}{C^{4 + \alpha}(C\eta^\alpha)^\alpha}\\
		&= \left(\left(\frac{\mu_+}{\mu_-}\right)^{\frac{N(\varepsilon)}{k(\varepsilon)}}\frac{1}{(1 + \delta)^4}\right)^{k(\varepsilon)}\frac{\eta^{1 - \alpha^2}}{C^{4 + 2\alpha}}\\
		& > 1.
	\end{align*}
	
	This is a contradiction, thus the proof is completed.
\end{proof}
	
\subsection{Bootstrap of the topological conjugacy}

In this subsection, we show Corollary \ref{intro cor: rigidity of conjugacy} which is a typical rigidity phenomenon in smooth dynamics: \textit{weak equivalence} (topological conjugacy)  \textit{implies strong equivalence}(conjugacy smooth along stable foliation). 

The following theorem is a general rigidity result on the universal cover in which all we need is one-dimensional stable bundle. Recall that $H$ is a conjugacy between $F$ and $\Psi$ on $N$.

\begin{theorem}\label{thm: pdc implies H smooth}
	If $f$ is $C^r$-smooth ($r > 1$) and $\lambda^s(p, f) = \lambda^s(q, f)$, $\forall p$, $q\in\text{Per}(f)$, then $H$ is $C^r$-smooth along each stable leaf.
\end{theorem}

Combining Theorem \ref{thm: conjugacy implies pdc} and Theorem \ref{thm: pdc implies H smooth}, we can get Corollary \ref{intro cor: rigidity of conjugacy}, which is restated as follows.

\begin{corollary}\label{cor: rigidity of conjugacy}
	Let $f$ be a $C^r$ ($r > 1$) Anosov map on a nilmanifold $M$ with one-dimensional stable bundle and totally non-invertible linear part $\Psi$. If $f$ is topologically conjugate to $\Psi$, then the conjugacy is $C^r$-smooth along each stable leaf.
\end{corollary}

\begin{proof}
	By Theorem \ref{thm: conjugacy implies pdc} and Theorem \ref{thm: pdc implies H smooth}, $H$ is $C^r$-smooth along each stable leaf. By Proposition \ref{prop: unique H}, $H$ commutes with $\Gamma$ and is projected to a conjugacy between $f$ and $\Psi$. Moreover, from the proof, for any conjugacy $\widehat{h}$ between $f$ and $\Psi$ and a lift $\widehat{H}$ of it, we have $H = (Ad_n\circ A)^{-1}\circ R_n\circ\widehat{H}$ for some $n\in N$ and $A\in\text{Aut}(N)$, thus $\widehat{H}$ is $C^r$-smooth along each stable leaf. Finally, $\widetilde{\mathcal{F}}^s$ is projected to $\mathcal{F}^s$, so $\widehat{h}$ is $C^r$-smooth along each stable leaf, too.
\end{proof}

To prove Theorem \ref{thm: pdc implies H smooth}, we need some preparations. As usual \cite{GG2008,An2023}, we construct an affine metric on each stable leaf by Livschitz Theorem \cite{L72} which is a main tool in rigidity issue about Anosov systems. For convenience, we state Livschitz Theorem in the same way as \cite{Katok1995} and give some explanation about the regularity of the solution of  cohomology equation.

\begin{proposition}\label{prop: Livschitz}
	{\rm (\cite{Katok1995}, Theorem 19.2.1)} Let $f$ be a $C^{1 + \alpha}$ ($0 < \alpha \leq \text{Lip}$) transitive Anosov map on a closed Riemannian manifold $M$, and $\phi: M \to \mathbb{R}$ be an $\alpha$-H\"older continuous function on $M$. Assume that $\sum_{i = 0}^{N_p - 1}\phi(f^i(p)) = 0$ for every periodic point $p\in\text{Per}(f)$ with period $N_p$, then there exists a unique (up to additive constant) $\alpha$-H\"older continuous function $\varphi: M \to \mathbb{R}$ satisfying $\phi = \varphi\circ f - \varphi$.
\end{proposition}

\begin{remark}\label{rmk: regularity of Livschitz}
	Generally, if $f\in C^r$ ($r > 1$) and $\phi \in C^{r - 1}$, then $\varphi$ is $C^{r - 1}$-smooth along each stable leaf. To see this, assume that $x\in N$ and $y\in\widetilde{\mathcal{F}}^s(x)$, then $\varphi(x) - \varphi(y) = \sum_{i = 0}^{+\infty}(\phi(f^i(y)) - \phi(f^i(x)))$. For fixed $x$, $\phi(f^i(y)) - \phi(f^i(x)) \in C^{r - 1}$, and the derivatives of each order at $x$ along $\widetilde{\mathcal{F}}^s(x)$ converges to zero exponentially as $i\to +\infty$. For example, take a parameterization $r: (-\varepsilon, \varepsilon) \to \widetilde{\mathcal{F}}^s(x)$, $r(0) = x$, we have $\left.\frac{d}{dt}\right|_{t = 0}\phi(f^i(\gamma(t))) = D_{f^i(x)}\phi(D_xf^i(\gamma'(0)))$, whereas $\gamma'(0)$ is a stable vector. Consequently, $\varphi$ is $C^{r - 1}$-smooth along each stable leaf.
\end{remark}

Recall that Anosov maps on nilmanifolds are transitive (see Lemma \ref{lem: transitivity}), so Proposition \ref{prop: Livschitz} is valid under the assumptions of Theorem \ref{thm: pdc implies H smooth}.

\begin{corollary}
	Under the assumptions of Theorem \ref{thm: pdc implies H smooth}, there is a function $\varphi: M \to \mathbb{R}$, $C^{r - 1}$-smooth along each stable leaf, such that $\ln\|Df|_{E^s}\| - \lambda^s(\Psi) = \varphi\circ f - \varphi$.
\end{corollary}

\begin{proof}
	By Theorem \ref{thm: p, q = A}, for every $p\in\text{Per}(f)$ with period $N_p$, we have
	\[\sum_{i = 0}^{N_p - 1}\ln\left\|Df|_{E^s(f^i(p))}\right\| = \ln\mu^s(\Psi)^{N_p} = N_p\lambda^s(\Psi),\]
	hence the condition of Proposition \ref{prop: Livschitz} is satisfied for $\phi = \ln\|Df|_{E^s}\| - \lambda^s(\Psi)$.
\end{proof}

Denote $\widetilde{\varphi} = \varphi\circ\pi: N \to\mathbb{R}$. For $x\in N$ and $y\in\widetilde{\mathcal{F}}^s(x)$, let $r : [0, 1] \to \widetilde{\mathcal{F}}^s(x)$ be a $C^1$ curve in $\widetilde{\mathcal{F}}^s(x)$ satisfying $r(0) = x$ and $r(1) = y$. Define
\[d_s(x, y) := \int_0^1 e^{\widetilde{\varphi}(t)}\left\|r'(t)\right\|dt.\]
Clearly the definition does not depend on the choice of $r$ since $\dim\widetilde{\mathcal{F}}^s = 1$. Moreover, $d_s$ is a continuous distance and is $C^r$-smooth along each stable leaf.

By the definition, one can check that $d_s(\cdot,\cdot)$ is an affine metric and invariant under the holonomy induced by unstable foliation. We refer to \cite[Proposition 4.4]{An2023} for a complete parallel proof on torus. 

\begin{lemma}\label{lem: affine metric}
	The distance $d_s(\cdot,\cdot)$ on each leaf of $\widetilde{\mathcal{F}}^s$ satisfies the followings.
	\begin{enumerate}
		\item $d_s$ is equivalent with $d_{\widetilde{\mathcal{F}}^s}$ and also right-$\Gamma$-invariant;
		\item $d_s(F(x), F(y)) = \mu^s(\Psi)d_s(x, y)$, $\forall x\in N$, $y\in\widetilde{\mathcal{F}}^s(x)$;
		\item $d_s(x, \beta_{\widetilde{\mathcal{F}}}(x, y)) = d_s(y, \beta_{\widetilde{\mathcal{F}}}(y, x))$, $\forall x$, $y\in N$, where $\beta_{\widetilde{\mathcal{F}}}(\cdot,\cdot)$ is given by Lemma \ref{lem: global product structure}.
	\end{enumerate}
\end{lemma}

Note that for the algebraic case, $\phi = 0$, $\varphi = 0$, $d_s = d_{\widetilde{\mathcal{L}}^s}$.

Recall the proof of Lemma \ref{lem: Holder}. If we replace $d_{\widetilde{\mathcal{F}}^s}$ by $d_s$, then $\mu_{\pm}^s(F)$ in the proof can be replaced by $\mu^s(\Psi)$, and the H\"older exponents of $H$ and $H^{-1}$ equal to 1, which means $H$ is Lipschitz continuous along $\widetilde{\mathcal{F}}^s$, and $H^{-1}$ is Lipschitz continuous along $\widetilde{\mathcal{L}}^s$.

\begin{lemma}\label{lem: H is isometric under affine metric}
	By scaling $d_s$ properly, $H$ is isometric along $\widetilde{\mathcal{F}}^s$, and $H^{-1}$ is isometric along $\widetilde{\mathcal{L}}^s$.
\end{lemma}

\begin{proof}
	It suffices to prove the conclusion only for $H^{-1}$. In fact, we prove that $H^{-1}$ is differentiable along $\widetilde{\mathcal{L}}^s$ everywhere and the derivative equals to 1. That is,
	$$\frac{d_s(H^{-1}(y), H^{-1}(z))}{d_{\widetilde{\mathcal{L}}^s}(y, z)} \to 1, \ \forall y\in N, \ z\in\widetilde{\mathcal{L}}^s(y), \ z \to y.$$
	
	Since $H^{-1}$ is Lipschitz along $\widetilde{\mathcal{L}}^s$, there exists $x\in N$ such that $H^{-1}$ is differentiable along $\widetilde{\mathcal{L}}^s(x)$ at $x$. By scaling $d_s$ properly, assume that the derivative equals to 1. Then for any $\varepsilon > 0$ there exists $\delta > 0$ such that if $x'\in \widetilde{\mathcal{L}}^s(x)$ and $d_{\widetilde{\mathcal{L}}^s}(x', x) < \delta$, then
	$$\left|\frac{d_s(H^{-1}(x'), H^{-1}(x))}{d_{\widetilde{\mathcal{L}}^s}(x', x)} - 1\right| < \frac{\varepsilon}{2}.$$
	\begin{claim*}
		For any fixed $y\in N$ and $z\in\widetilde{\mathcal{L}}^s(y)$, there exists $x_k\in\widetilde{\mathcal{L}}^s(x)$ for $k \geq 1$, such that as $k \to +\infty$,
		\begin{enumerate}
			\item $d_{\widetilde{\mathcal{L}}^s}(x_k, x) \to d_{\widetilde{\mathcal{L}}^s}(z, y)$;
			\item $d_s(H^{-1}(x_k), H^{-1}(x)) \to d_s(H^{-1}(z), H^{-1}(y))$.
		\end{enumerate}
	\end{claim*}
	\begin{proof}[Proof of Claim]
		Since $\Gamma$ is cocompact, there exists some compact subset $K \subseteq N$ such that $K\Gamma = N$. Denote $C := \max\{d_{\widetilde{\mathcal{L}}^s}(u, \beta(u, v)): u, v\in K\} < +\infty$.
		
		For $k\geq 1$, consider two points $\Psi^{-k}(x)$ and $\Psi^{-k}(y)$. By right-invariance of $d$ and $\beta$, there exists $\gamma'_k\in\Gamma$ such that $d_{\widetilde{\mathcal{L}}^s}(\Psi^{-k}(y)\gamma'_k, \beta(\Psi^{-k}(y)\gamma'_k, \Psi^{-k}(x))) \leq C$.
		
		Denote $\gamma_k = \Psi^k(\gamma'_k)$, and $y'_k = \beta(\Psi^{-k}(y)\gamma'_k, \Psi^{-k}(x))$, then $y_k := \Psi^n(y'_k) = \beta(y\gamma_k, x)$, and hence $d_{\widetilde{\mathcal{L}}^s}(y_k, y\gamma_k) \leq \mu^s(\Psi)^{k - 1}d_{\widetilde{\mathcal{L}}^s}(y_1, y\gamma_1) \to 0$ as $k \to +\infty$.
		
		Define $x_k := \beta(x, z\gamma_k)$. By Lemma \ref{lem: affine metric}, we have
		\begin{align*}
            d_{\widetilde{\mathcal{L}}^s}(x_k ,x)
            &= d_{\widetilde{\mathcal{L}}^s}(x, \beta(x, z\gamma_k)) = d_{\widetilde{\mathcal{L}}^s}(z\gamma_k, \beta(z\gamma_k, x))\\
            &= d_{\widetilde{\mathcal{L}}^s}(z\gamma_k, \beta(y\gamma_k, x)) = d_{\widetilde{\mathcal{L}}^s}(z\gamma_k, y_k).
        \end{align*}
		
		Since $d_{\widetilde{\mathcal{L}}^s}(y_k, y\gamma_k) \to 0$ as $k \to +\infty$, $d_{\widetilde{\mathcal{L}}^s}(y\gamma_k, z\gamma_k) = d_{\widetilde{\mathcal{L}}^s}(y, z)$, it follows that $d_{\widetilde{\mathcal{L}}^s}(x_k, x) \to d_{\widetilde{\mathcal{L}}^s}(y, z)$ as $k \to +\infty$, and 1 is proved.
		
		For 2, notice that $H^{-1}\beta(u, v) = \beta_{\widetilde{\mathcal{F}}}(H^{-1}(u), H^{-1}(v))$, thus by 3 of Lemma \ref{lem: affine metric} and the same argument when proving $d_{\widetilde{\mathcal{L}}^s}(x_k, x) = d_{\widetilde{\mathcal{L}}^s}(z\gamma_k, y_k)$, we have
		$$d_s(H^{-1}(x_k), H^{-1}(x)) = d_s(H^{-1}(z\gamma_k), H^{-1}(y_k)).$$
		Moreover, by uniform continuity of $H^{-1}$ and Lemma \ref{lem: sequence}, we also have
        \begin{align*}
            &d_s(H^{-1}(y_k), H^{-1}(y\gamma_k)) \to 0\\
            &d(H^{-1}(y\gamma_k), H^{-1}(y)\gamma_k) \to 0\\
            &d(H^{-1}(z\gamma_k), H^{-1}(z)\gamma_k) \to 0
        \end{align*}
        Hence $d_s(H^{-1}(x_k), H^{-1}(x)) \to d_s(H^{-1}(y)\gamma_k, H^{-1}(z)\gamma_k) = d_s(H^{-1}(y), H^{-1}(z))$.
	\end{proof}
	Now fix any $y\in N$ and $z\in\widetilde{\mathcal{L}}^s(y)$ satisfying $d_{\widetilde{\mathcal{L}}^s}(z, y) < \delta$, let $\{x_k\}$ be the sequence constructed in the claim. Then for $k$ sufficiently large, we have $d_{\widetilde{\mathcal{L}}^s}(x_k, x) < \delta$. Hence
	$$\left|\frac{d_s(H^{-1}(x_k), H^{-1}(x))}{d_{\widetilde{\mathcal{L}}^s}(x_k, x)} - 1\right| < \frac{\varepsilon}{2}.$$
	By the claim, we can require $n$ to be larger, such that
	$$\left|\frac{d_s(H^{-1}(x_k), H^{-1}(x))}{d_{\widetilde{\mathcal{L}}^s}(x_k, x)} - \frac{d_s(H^{-1}(z), H^{-1}(y))}{d_{\widetilde{\mathcal{L}}^s}(z, y)}\right| < \frac{\varepsilon}{2}.$$
	It follows that
	$$\left|\frac{d_s(H^{-1}(z), H^{-1}(y))}{d_{\widetilde{\mathcal{L}}^s}(z, y)} - 1\right| < \varepsilon,$$
	and this completes the proof.
\end{proof}

Now $H$ and $H^{-1}$ has constant derivative along stable leaves, under $d$ and $d_s$. Since $d_s$ is $C^r$-smooth along each stable leaf, it follows that $H$ and $H^{-1}$ is $C^r$-smooth along each stable leaf under $d$, and the proof of Theorem \ref{thm: pdc implies H smooth} is completed.

\subsection{The same stable Lyapunov exponents implies  conjugacy}

In this subsection we prove Theorem \ref{intro thm: pdc implies conjugacy}, the opposite direction of Theorem \ref{thm: conjugacy implies pdc}. For convenience, we restate it as follow.

\begin{theorem}\label{thm: pdc implies conjugacy}
	If $\ \Psi$ is horizontally irreducible and $f\in C^r$ ($r > 1$) with $\lambda^s(p, f) = \lambda^s(q, f)$, for all $p$, $q\in\text{Per}(f)$,  then $f$ is topologically conjugate to $\Psi$.
\end{theorem}

Recall that $\Psi$ is {\it horizontally irreducible} means that the horizontal part of $\Psi$ is irreducible. We introduce this condition to ensure that $\widetilde{\mathcal{L}}^s(\Gamma)$ is dense in $N$.

\begin{lemma}\label{lem: dense leaf}
	{\rm (\cite{DeWitt2021}, Lemma 6.3)} Let $\mathfrak{h}$ be a one-dimensional subalgebra in $\mathfrak{n}$. Then $(\exp \mathfrak{h})\Gamma$ is dense in $N/\Gamma$ if and only if $(\exp \mathfrak{h})N_2\Gamma$ is dense in $N/N_2\Gamma$.
\end{lemma}

If $\Psi$ is horizontally irreducible, then $\Psi_1 : N/N_2\Gamma \to N/N_2\Gamma$ is irreducible, hence its stable foliation is minimal. Notice that the stable leaf of $\Psi_1$ at $eN_2\Gamma$ is actually $(\exp\mathfrak{n}^s)N_2\Gamma$, since $\Psi$ has one-dimensional stable bundle. By Lemma \ref{lem: dense leaf}, $(\exp\mathfrak{n}^s)\Gamma$ is dense in $N/\Gamma$, hence $\widetilde{\mathcal{L}}^s(\Gamma)$ is dense in $N$.

\begin{proof}[Proof of Theorem \ref{thm: pdc implies conjugacy}]
	It suffices to show that the conjugacy $H$ commutes with $\Gamma$, or equivalently, $H^{-1}$ commutes with $\gamma$ for all $\gamma\in \Gamma$.
    
    By Lemma \ref{lem: H and Gamma}, $H^{-1}(x\gamma)\gamma^{-1} \in \widetilde{\mathcal{F}}^s(H^{-1}(x))$. To show that $H^{-1}(x\gamma)\gamma^{-1} = H^{-1}(x)$, it suffices to show that $\alpha(x, \gamma) := d_s(H^{-1}(x\gamma)\gamma^{-1}, H^{-1}(x)) = 0$. Clearly $\alpha(\cdot,\ \cdot)$ is continuous.
	\begin{claim*}
		There exist an order $<_{\mathcal{L}}$ for each leaf of $\widetilde{\mathcal{L}}^s$ and an order $<_{\mathcal{F}}$ for each leaf of $\widetilde{\mathcal{F}}^s$, such that the followings hold. We omit the subscript when there is no confusion.
		\begin{enumerate}
			\item Continuity: if $x$, $y$ lie in the same leaf, $x < y$, then there are neighborhoods $U_x$ of $x$ and $U_y$ of $y$, such that for any $u_x\in U_x$, $u_y\in U_y$, $u_x < u_y$ as long as they lie in the same leaf;
			\item Right-invariance for $<_{\mathcal{L}}$: if $y\in \widetilde{\mathcal{L}}^s(x)$, and $x < y$, then $xn < yn$, $\forall n\in N$;
			\item Right-$\Gamma$-invariance for $<_{\mathcal{F}}$: if $y\in \widetilde{\mathcal{F}}^s(x)$, and $x < y$, then $x\gamma < y\gamma$, $\forall \gamma\in\Gamma$.
		\end{enumerate}
	\end{claim*}
	\begin{proof}[Proof of Claim]
		Take a nonzero vector $v\in\mathfrak{n}^s$, and define an order in $\mathfrak{n}^s$ by identifying it with $\mathbb{R}$. Then one can define an order in every leaf of $\widetilde{\mathcal{L}}^s$ by the exponential map and right translation, satisfying 1 and 2.
		
		Further, one can define an order in every leaf of $\widetilde{\mathcal{F}}^s$ by $H^{-1}$, satisfying 1.
		
		To show that such an order also satisfies 3, it suffices to show that for any fixed $x\in N$, the right translation $R_\gamma : \widetilde{\mathcal{F}}^s(H^{-1}(x)) \to \widetilde{\mathcal{F}}^s(H^{-1}(x)\gamma)$ is order preserving. By definition of $<_{\mathcal{F}}$, it suffices to show that $H\circ R_\gamma\circ H^{-1}: \widetilde{\mathcal{L}}^s(x) \to \widetilde{\mathcal{L}}^s(H(H^{-1}(x)\gamma))$ is order preserving.
		
		Notice that $H(H^{-1}(x)\gamma) \in \widetilde{\mathcal{L}}^s(x\gamma)$ by Lemma \ref{lem: H and Gamma}, hence $H\circ R_\gamma\circ H^{-1}: \widetilde{\mathcal{L}}^s(x) \to \widetilde{\mathcal{L}}^s(x\gamma)$. Since $<_{\mathcal{L}}$ is right-invariant, it suffices to show that $R_\gamma^{-1}\circ H\circ R_\gamma\circ H^{-1}: \widetilde{\mathcal{L}}^s(x) \to \widetilde{\mathcal{L}}^s(x)$ is order preserving. Now $d(R_\gamma^{-1}\circ H\circ R_\gamma\circ H^{-1}, \text{\rm Id}_N) < +\infty$, it must be an order-preserving homeomorphism.
	\end{proof}
	\begin{claim*}
		Fix $\gamma\in\Gamma$. If $\alpha(x, \gamma) = \alpha(y, \gamma)$, $\forall x, y\in N$, then $\alpha(x, \gamma) = 0$, $\forall x\in N$.
	\end{claim*}
	\begin{proof}[Proof of Claim]
		Assume for contradiction that $\alpha(x, \gamma) = \alpha > 0$, $\forall x\in N$. First notice that $x_k := H^{-1}(x\gamma^k)\gamma^{-k}$ lies in $\widetilde{\mathcal{F}}^s(H^{-1}(x))$ for $k\in\mathbb{Z}$, which is one-dimensional. Besides,
		\begin{align*}
		    d_s(x_{k + 2}, x_{k + 1})
            &= d_s(x_{k + 2}\gamma^{k + 1}, x_{k + 1}\gamma^{k + 1}) = \alpha(x\gamma^{k + 1}, \gamma)\\
            &= \alpha\\
            &= \alpha(x\gamma^k, \gamma) = d_s(x_{k + 1}, x_k).
		\end{align*}
		Therefore, either there exists $k\in\mathbb{Z}$ such that $x_{k + 2} = x_k$, or $d_s(x_0, x_k) = |k|\alpha$, $\forall k \in\mathbb{Z}$.
		
		The latter case causes contradiction. To see this, we have by Lemma \ref{lem: quasi-isometry} and Lemma \ref{lem: affine metric} that $d_s(x, y) \leq Cd(x, y) + C$ for some constant $C > 0$. Now $d_s(x_0, x_k) = k\alpha\to +\infty$ as $k\to+\infty$, so $d(x_0, x_k) \to +\infty$, which actually means that $d(H^{-1}(x), H^{-1}(x\gamma^k)\gamma^{-k})\to +\infty$.
		
		However, $d(H^{-1}(x), H^{-1}(x\gamma^k)\gamma^{-k}) \leq d(H^{-1}(x), x) + d(x\gamma^k, H^{-1}(x\gamma^k)) \leq 2C_0$.
		
		The former case also causes contradiction. To see this, notice that $d_s(x_k, x_{k + 2})$ is a continuous function with respect to $x$, and take discrete values 0 and $2\alpha$. Now it takes the value 0 at $x$, so it must be zero function, and thus $\alpha(x, \gamma^2) = 0$, $\forall x\in N$.
		
		By continuity, $\{x\in N: H^{-1}(x) < H^{-1}(x\gamma)\gamma^{-1}\}$ and $\{x\in N: H^{-1}(x) > H^{-1}(x\gamma)\gamma^{-1}\}$ are both open subsets of $N$. Now that $H^{-1}(x)\not=H^{-1}(x\gamma)\gamma^{-1}$, and $N$ is connected, one of them must be $N$. Without loss of generality, assume $H^{-1}(x) < H^{-1}(x\gamma)\gamma^{-1}, \forall x\in N$. Then by the right-$\Gamma$-invariance of the order, one has
		$$H^{-1}(x) < H^{-1}(x\gamma)\gamma^{-1} < (H^{-1}((x\gamma)\gamma)\gamma^{-1})\gamma^{-1} = H^{-1}(x\gamma^2)\gamma^{-2},$$
		which contradicts with $\alpha(x, \gamma^2) = 0$.
		
		Both case leads to contradiction, and the proof is completed.
	\end{proof}
	By the claim, now it suffices to show that $\alpha(x, \gamma) = \alpha(y, \gamma)$, $\forall x, y\in N$, $\forall\gamma\in\Gamma$.
	
	First we show that $\alpha(x, \gamma) = \alpha(z, \gamma)$, $\forall x\in N$, $\forall z\in\widetilde{\mathcal{L}}^s(x)$, $\forall \gamma\in\Gamma$. Notice that $H^{-1}(x)$, $H^{-1}(z)$, $H^{-1}(x\gamma)\gamma^{-1}$, $H^{-1}(z\gamma)\gamma^{-1}$ all lie in $\widetilde{\mathcal{F}}^s(x)$. Since $H^{-1}$ and $R_\gamma^{-1}\circ H^{-1}\circ R_\gamma$ are both order-preserving, we may assume that $x < z$ and hence $H^{-1}(x) < H^{-1}(z)$, $H^{-1}(x\gamma)\gamma^{-1} < H^{-1}(z\gamma)\gamma^{-1}$. Moreover, since $\widetilde{\mathcal{F}}^s(x)$ is right-$\Gamma$-invariant and $H^{-1}$ is isometric along stable leaves under $d_{\widetilde{\mathcal{L}}^s}$ and $d_s$, we have
	\begin{align*}
		d_s(H^{-1}(x\gamma)\gamma^{-1}, H^{-1}(z\gamma)\gamma^{-1})
		&= d_s(H^{-1}(x\gamma), H^{-1}(z\gamma)) = d_{\widetilde{\mathcal{L}}^s}(x\gamma, z\gamma)\\
		&= d_{\widetilde{\mathcal{L}}^s}(x, z) = d_s(H^{-1}(x), H^{-1}(z)).
	\end{align*}
	Hence there are only four cases:
	\begin{align*}
		&H^{-1}(x) < H^{-1}(z) \leq H^{-1}(x\gamma)\gamma^{-1} < H^{-1}(z\gamma)\gamma^{-1}\\
		&H^{-1}(x) \leq H^{-1}(x\gamma)\gamma^{-1} < H^{-1}(z) \leq H^{-1}(z\gamma)\gamma^{-1}\\
		&H^{-1}(x\gamma)\gamma^{-1} < H^{-1}(x) < H^{-1}(z\gamma)\gamma^{-1} < H^{-1}(z)\\
		&H^{-1}(x\gamma)\gamma^{-1} < H^{-1}(z\gamma)\gamma^{-1} \leq H^{-1}(x) < H^{-1}(z).
	\end{align*}
	
	In any case, $d(H^{-1}(x\gamma)\gamma^{-1}, H^{-1}(x)) = d(H^{-1}(z\gamma)\gamma^{-1}, H^{-1}(z))$, i.e., $\alpha(x, \gamma) = \alpha(z, \gamma)$.
	
	Next we show that $\alpha(x, \gamma) = \alpha(e, \gamma)$, $\forall x\in N$, $\forall\gamma\in\Gamma$. This actually completes the proof.
	
	Since $\widetilde{\mathcal{L}}^s(\Gamma)$ is dense in $N$, we have that $\Psi^k(\widetilde{\mathcal{L}}^s(\Gamma)) = \widetilde{\mathcal{L}}^s(\Psi^k\Gamma)$ is dense in $N$, $\forall k \geq 1$. Thus for any fixed $x\in N$, there exists $z_k\in\widetilde{\mathcal{L}}^s(e)$ and $\gamma_k\in\Psi^k\Gamma$ such that $z_k\gamma_k \to x$ as $k\to +\infty$.
	
	By uniform continuity of $H^{-1}$, we have
    \begin{align*}
        &d(H^{-1}(z_k\gamma_k), H^{-1}(x)) \to 0,\\
        &d(H^{-1}(z_k\gamma_k\gamma), H^{-1}(x\gamma)) \to 0.
    \end{align*}
    Therefore, $d_s(H^{-1}(z_k\gamma_k\gamma), H^{-1}(z_k\gamma_k)\gamma) \to \alpha(x, \gamma)$.
	
	By Lemma \ref{lem: sequence}, we have
    \begin{align*}
        &d(H^{-1}(z_k\gamma_k), H^{-1}(z_k)\gamma_k) \to 0,\\
        &d(H^{-1}(z_k\gamma_k\gamma), H^{-1}(z_k[\gamma_k, \gamma]\gamma)\gamma_k) \to 0.
    \end{align*}
    Therefore, by right-$\Gamma$-invariance of $d_s$,
	\begin{align*}
		&d_s(H^{-1}(z_k[\gamma_k, \gamma]\gamma)\gamma_k, H^{-1}(z_k)\gamma_k\gamma) \to \alpha(x, \gamma),\\
		&d_s(H^{-1}(z_k[\gamma_k, \gamma]\gamma), H^{-1}(z_k)[\gamma_k, \gamma]\gamma) \to \alpha(x, \gamma),\\
		&\alpha(e, [\gamma_k, \gamma]\gamma) = \alpha(z_k, [\gamma_k, \gamma]\gamma) \to \alpha(x, \gamma).
	\end{align*}
	
	Now prove $\alpha(x, \gamma) = \alpha(e, \gamma)$ inductively. First assume that $\gamma\in\Gamma_s$, where $s$ is the step of nilpotency of $N$. Then $[\gamma_k, \gamma] = e$, $\alpha(e, \gamma) \to \alpha(x, \gamma)$ and thus the conclusion holds. Now assume that the conclusion holds for $\gamma\in\Gamma_{i + 1}$, and we shall prove that $\alpha(x, \gamma) = \alpha(e, \gamma)$ for $\gamma\in\Gamma_i$. By the inductive hypothesis and claim, $\alpha(x, \gamma) = 0$, $\forall x\in N$, $\forall \gamma\in\Gamma_{i + 1}$, which means $H^{-1}(x\gamma) = H^{-1}(x)\gamma$, $\forall x\in N$, $\forall\gamma\in\Gamma_{i + 1}$. Now $\gamma\in\Gamma_i$, hence $[\gamma^{-1}, \gamma_k]\in\Gamma_{i + 1}$, and we have
	\begin{align*}
		\alpha(e, [\gamma_k, \gamma]\gamma)
		& = d_s(H^{-1}(e\gamma_k\gamma\gamma_k^{-1}), H^{-1}(e)\gamma_k\gamma\gamma_k^{-1})\\
		& = d_s(H^{-1}(e\gamma[\gamma^{-1}, \gamma_k]), H^{-1}(e)\gamma[\gamma^{-1}, \gamma_k])\\
		& = d_s(H^{-1}(e\gamma)[\gamma^{-1}, \gamma_k], H^{-1}(e)\gamma[\gamma^{-1}, \gamma_k])\\
		& = d_s(H^{-1}(e\gamma), H^{-1}(e)\gamma) = \alpha(e, \gamma).
	\end{align*}
	Therefore $\alpha(e, \gamma) \to \alpha(x, \gamma)$ and thus $\alpha(x, \gamma) = \alpha(e, \gamma)$, the induction is completed.
\end{proof}

Finally, summarizing the above results, we can prove Corollary \ref{intro cor: equivalence}.
\begin{proof}[Proof of Corollary \ref{intro cor: equivalence}]
	Combining with Theorem \ref{thm: p, q = A}, Theorem \ref{thm: special and conjugacy}, Theorem \ref{thm: conjugacy implies pdc}, Theorem \ref{thm: pdc implies conjugacy} and Corollary \ref{cor: rigidity of conjugacy}, we get the equivalence between $(1)$, $(2)$, $(3)$, $(4)$ of Corollary \ref{intro cor: equivalence} and the regularity of conjugacy restricted on each stable leaf. Since $(5) \implies (1)$ is trivial, we need only prove that $(4)\implies (5)$. Let $f$ be $C^{r + 1}$-smooth. By Lemma \ref{lem: affine metric}, the holonomies induced by unstable foliation are actually translations between stable leaves under the metric $d_s(\cdot,\cdot)$. Recall that $d_s(\cdot,\cdot)$ is $C^{r + 1}$-smooth by Remark \ref{rmk: regularity of Livschitz}. It follows that the holonomies induced by unstable foliations are $C^{r + 1}$-smooth. Hence $\mathcal{F}^u$ is in fact a $C^{r_* + 1}$-smooth foliation and $E^u$ is a $C^{r_*}$-smooth distribution. We refer to \cite[Section 6]{PSW1997} for more details about the relationship among the regularity of honolomy, foliation and bundle.
\end{proof}

\section*{Statements and Declarations}

\begin{itemize}
\item Funding: Wenchao Li is partially supported by National Key R\&D Program of China (2022YFA 1005801) and NSFC 12161141002.
\item Competing interests: The authors have no competing interests to declare that are relevant to the content of this article.
\item Ethics approval and consent to participate: Not applicable.
\item Consent for publication: Not applicable.
\item Data availability: Not applicable.
\item Materials availability: Not applicable.
\item Code availability: Not applicable.
\item Author contribution: Ruihao Gu wrote the introduction part of this manuscript in consultation with Wenchao Li. Wenchao Li wrote the rest of this manuscript in consultation with Ruihao Gu.
\end{itemize}

\bibliographystyle{alpha}
\bibliography{Stable_LSR_of_Nil-Endo}

\end{document}